\documentclass{amsart}[11pt]
\usepackage[left=1.25in, right=1.25in, top=1.25in, bottom=1.25in]{geometry}
\usepackage{amsmath,amsfonts,amssymb,latexsym,amsthm}
\usepackage{epsfig}
\usepackage{hyperref} 
\usepackage[dvipsnames]{xcolor}
\usepackage{seqsplit}

\usepackage{graphicx}
\usepackage{amsfonts}
\usepackage{eucal}
\usepackage{amscd}
\usepackage{amsthm}
\usepackage{amsmath}
\usepackage{amssymb}
\usepackage{comment}
\usepackage{mathtools}
\usepackage{tikz}
\usepackage{ytableau}
\usepackage{scalerel}
\usepackage{csquotes}
\usepackage{tabularx}
\usepackage{enumerate}

\numberwithin{equation}{section}

\newtheorem{thm}{Theorem}[section]
\newtheorem{lemma}[thm]{Lemma}

\newtheorem{cor}{Corollary}[thm]
\newtheorem{prop}[thm]{Proposition}
\newtheorem{conj}[thm]{Conjecture}

\newtheorem{ex}[thm]{Example}

\newtheorem{defin}[thm]{Definition}

\usepackage[colorinlistoftodos]{todonotes}

\RequirePackage{cleveref}
\usepackage{hypcap}
\hypersetup{colorlinks=true, citecolor=darkblue, linkcolor=darkblue}
\definecolor{darkblue}{rgb}{0.0,0,0.7}

\definecolor{darkred}{rgb}{0.68,0,0}

\definecolor{darkgreen}{rgb}{0,.38,0}

\def\<{\langle}
\def\>{\rangle}

\def\0{{\mathbf 0}}

\def\.{\hskip.06cm}

\def\msquare{\mathord{\scalerel*{\Box}{gX}}}
\usepackage[
backend=biber,
style=alphabetic
]{biblatex}
\title[Kronecker product of Schur functions of square shapes]{On the Kronecker product of Schur functions of square shapes}

\author{Chenchen Zhao}
\address[Zhao]{Department of Mathematics, University of Southern California, Los Angeles, CA 90089} 
\email{zhao109@usc.edu} 
\date{\today}

\begin{document}

\begin{abstract}
Motivated by the Saxl conjecture and the tensor square conjecture, which states that the tensor squares of certain irreducible representations of the symmetric group contain all irreducible representations, we study the tensor squares of irreducible representations associated with square Young diagrams. We give a formula for computing Kronecker coefficients, which are indexed by two square partitions and a three-row partition, specifically one with a short second row and the smallest part equal to 1. We also prove the positivity of square Kronecker coefficients for particular families of partitions, including three-row partitions and near-hooks.
\end{abstract}

\maketitle

\section{Introduction}
Given partitions $\lambda, \mu \vdash n$, we can decompose the internal product of Schur functions as \[s_\lambda \ast s_\mu = \sum_{\nu \vdash n} g(\lambda,\mu,\nu) s_\nu,\]where $g(\lambda,\mu,\nu)$ are the Kronecker coefficients. The Kronecker coefficients can also be interpreted as the multiplicities of an irreducible module of $S_n$ in the tensor product of irreducible modules of $S_n$ corresponding to $\lambda$ and $\mu$. Therefore, the Kronecker coefficients are certainly non-negative integers, which naturally suggests that there may be a combinatorial interpretation of the coefficients. The problem of finding a non-negative combinatorial interpretation for the Kronecker coefficients was explicitly stated by Stanley in 2000 (\cite{stanley2000positivity} Problem 10) as a major open problem in Algebraic Combinatorics. The Kronecker coefficients have recently gained prominence within the context of algebraic complexity theory, particularly in the realm of Geometric Complexity Theory (GCT). However, as addressed by Panova in \cite{panova2023complexity}, despite the increasing interest in the problem, little progress has been made: The Kronecker product problem is still poorly understood, and deriving an explicit combinatorial formula to solve the Kronecker product remains as an outstanding open problem in the field of Algebraic Combinatorics.

The number of irreducible representations of the symmetric group $S_n$ is equal to the number of conjugacy classes, which is the number of integer partitions of $n$. Given $\mu \vdash n$, let $\mathbb S^\mu$ denote the Specht module of the symmetric group $S_n$, indexed by partition $\mu$. It is worth noting that these Specht modules provide us with a way to study the irreducible representations, with each representation being uniquely indexed by an integer partition (see e.g. \cite{books/daglib/0077285}).

In \cite{HSTZ}, Heide, Saxl, Tiep, and Zalesski proved that with a few exceptions, every irreducible character of a simple group of Lie type is a constituent of the tensor square of the Steinberg character. They conjectured that for $n\ge 5$, there is an irreducible character $\chi$ of $A_n$, whose tensor square $\chi \otimes \chi$ contains every irreducible character as a constituent. The following is the symmetric group analog of this conjecture:
\begin{conj}[Tensor Square Conjecture]
For every $n$ except 2,4,9, there exists an irreducible representation $V$ of the symmetric group $S_n$ such that the tensor square $V \otimes V$ contains every irreducible representation of $S_n$ as a summand with positive multiplicity. In terms of the correspondence of partitions, there exists a partition $\lambda \vdash n$ such that the Kronecker coefficient $g(\lambda,\lambda,\mu)$ is positive for any $\mu \vdash n$. 
\end{conj}

In 2012, Jan Saxl conjectured that all irreducible representations of $S_n$ over $\mathbb{C}$ occur in the decomposition of the tensor square of irreducible representation corresponding to the staircase shape partition \cite{pak2013kronecker}. This conjecture is as follows:
\begin{conj}[Saxl Conjecture]
Let $\rho_m$ denote the staircase partition of size $n := \binom{m+1}{2}$. Then $g(\rho_m,\rho_m,\mu) > 0$ for every $\mu \vdash n.$
\end{conj}

Previous work made progress towards the Tensor Square Conjecture, and specifically towards the Saxl Conjecture, see e.g. \cite{pak2013kronecker}, \cite{Ikenmeyer_2015}, \cite{Luo_2016}, \cite{LI2021112340}.
Attempts have also been made to understand the Kronecker coefficients from different aspects: combinatorial interpretations for some known special shapes, see e.g. \cite{Remmel1989AFF}, \cite{10.36045/bbms/1103408635},  \cite{Ballantine2005OnTK}, \cite{blasiak2012kronecker}, \cite{liu2014simplified}; from the perspective of the computational complexity of computing or deciding positivity of the Kronecker coefficients, see e.g. \cite{article}, \cite{pak2015complexity}, \cite{Ikenmeyer_2017}. 

In 2020, Bessenrodt and Panova made the following conjecture concerned with the shape of partitions satisfying the tensor square conjecture:
\begin{conj}[\cite{panova2023complexity}, Bessenrodt-Panova 2020]\label{conj:tensor_square_general}
For every $n$, there exists $k(n)$ such that the tensor square of every self-conjugate partition whose Durfee size is at least $k(n)$ and is not the $k\times k$ partition satisfies the Tensor Square Conjecture.
\end{conj}

In \cite{pak2013kronecker}, Pak, Panova, and Vallejo suggested that caret partitions may satisfy the tensor square conjecture. Many of the arguments on staircase shape could also be adapted for caret shapes and chopped-square shapes.


Most approaches to proving the positivity of a certain family of Kronecker coefficients use the semigroup property, see Section~\ref{background}, which relies on breaking the partition triple into smaller partitions. The minimal elements in this procedure are the rectangular shapes, and thus understanding Kronecker positivity in general starts from understanding Kronecker coefficients of rectangular shapes. 



In this paper, we study the tensor squares of irreducible representations corresponding to square Young diagrams, denoted $\msquare_m$. We show that the Kronecker coefficients $g(\msquare_m,\msquare_m,\mu)$ in the case where $\msquare_m = (m^m)$ has square shape and $\mu = (m^2-k, k-1, 1)$ vanish if and only if $k \le 4$ when $m\ge 5$. We give an explicit formula for $g(\msquare_m,\msquare_m,\mu)$ when $\mu = (m^2-k, k-1, 1)$ has a short second row: 
\begin{thm}[Theorem \ref{prop:zerocase}]
Let $f(k)$ be the number of partitions of $k$ with no parts equal to 1 or 2. Let $\ell_1(\alpha)$ denote the number of different parts of a partition $\alpha$. Then for $2\le k \le m$, \[g(\msquare_m,\msquare_m,(m^2-k,k-1,1)) = \sum_{\substack{\alpha \vdash k-1\\ \alpha_1=\alpha_2}} \ell_1(\alpha) - f(k).\]
\end{thm}

We completely prove the forward direction of the following conjecture and have partial work done towards the other direction, including showing the positivity of square Kronecker coefficients for three-row partitions and near-hooks.
\begin{conj}\label{conj:myconj} For $m \ge 7$, $ g(\msquare_m, \msquare_m,\mu) =0$ if and only if  $\mu \in S$ or $\mu' \in S$, where \[S := \{(m^2-3,2,1), (m^2-4,3,1), (m^2-j,1^j)\mid j \in \{1,2,4,6\}\}.\]\end{conj}

We state our main results as follows:
\begin{thm}[Corollary \ref{cor:zerocase}, \ref{zerocor}, Theorem \ref{near two-row}, \ref{mainthm}] For every integer $m\ge 7$, let $\mu \vdash m^2$ be a partition of length at most $3$, we have $g(\msquare_m, \msquare_m, \mu)>0$ if and only if $\mu \notin \{(m^2-3,2,1), (m^2-4,3,1), (m^2-2,1,1), (m^2-1,1)\}.$ \end{thm}

\begin{thm} [Corollary \ref{cor:conclusion-near-hooks}] Let $m$ be an integer and assume that $m \ge 20$. Define near-hook partitions $\mu_i(k,m) := (m^2-k-i,i, 1^k)$. Then for every $i \ge 8$, we have $g(\msquare_m,\msquare_m, \mu_i(k,m)) > 0$ for all $k \ge 0$.
\end{thm}

The rest of this paper is structured as follows. In Section~\ref{background}, we equip the reader with some required background information and notations. In Section \ref{section:missing}, we present the partitions that do not occur in tensor squares of square partitions. In Section \ref{sec:constituency} and Section \ref{sec:near-hooks}, we present the results on the positivity of square Kronecker coefficients for certain families of partitions. In Section \ref{sec:final}, we will discuss some additional remarks and related further research.

\subsection*{Acknowledgements}
The author would like to thank her advisor, Greta Panova, for suggesting the problem and for helpful guidance and insightful discussions throughout the project.

\section{Background}\label{background}
\subsection{Partitions}
A partition $\lambda$ of $n$, denoted as $\lambda \vdash n$, is a finite list of weakly decreasing positive integers a $(\lambda_1, \dots, \lambda_k)$ such that $\sum_{i =1}^k \lambda_i= n$. Given a partition $\lambda$, the size $|\lambda|$ is defined to be $\sum_{i =1}^k \lambda_i$. The length of $\lambda$ is defined to be the number of parts of the partition and we denote it by $\ell(\lambda).$ We use $P(n)$ to denote the set of all partitions of $n$.

We associate each partition $\lambda = (\lambda_1, \dots, \lambda_k)\vdash n$ with a Young diagram, which is a left justified array of $n$ boxes with $\lambda_i$ boxes in row $i.$ Denote by $\lambda'$ the conjugate (or transpose) of a partition $\lambda$. 
For instance, below are the Young diagrams corresponding to partition $\lambda = (5,3,2)$ and its transpose $\lambda' = (3,3,2,1,1)$.
\[
\ytableausetup {mathmode, boxsize=0.75em}
      \ydiagram{6,3,2}
      \ydiagram{2+3,2+3,2+2,2+1,2+1,2+1}
\]
The Durfee size of a partition $\lambda$, denoted by $d(\lambda)$ is the number of boxes on the main diagonal of the Young diagram of $\lambda.$  For the sake of convenience, we will refer to the irreducible representation corresponding to $\lambda$ be $\lambda.$
\begin{defin}
For $m \ge 1$, we define the square-shaped partition $\msquare_m \vdash m^2$ to be $\msquare_m := (m^m)$.\end{defin}

For $n \in \mathbb N$, we denote the symmetric group on $n$ symbols by $S_n$.  Let $\lambda, \mu \vdash n$. We say that $\lambda$ dominates $\mu$, denoted by $\lambda \unrhd \mu$, if $\sum_{i=1}^j \lambda_i \ge \sum_{i=1}^j \mu_i$ for all $j$. 

Let $p_k(a,b)$ denote the number of partitions of $k$ that fit into an $a \times b$ rectangle. We denote the number of partitions of $k$ that fit into an $m\times m$ square by $P_k(m).$ Note that $P_k(m) = p_k(m,m).$

Given $\mu \vdash n$, let $\chi^{\mu}$ denote the irreducible character of the symmetric group $S_{\mu}$ and let $\chi^{\mu}[\alpha]$ denote the value of $\chi^\mu (\omega)$ on any permutation $\omega$ of cycle type $\alpha.$ The characters can be computed using the Murnaghan-Nakayama Rule (see e.g. \cite{stanley1997enumerative} for more details about the rule).
\begin{thm}[Murnaghan-Nakayama Rule] We have \[\chi^{\lambda}(\alpha) = \sum_T(-1)^{ht(T)},\] summed over all border-strip tableaux of shape $\lambda$ and type $\alpha$ and $ht(T)$ is the sum of the heights of each border-strip minus $\ell(\alpha)$. 
\end{thm}

\subsection{The Kronecker coefficients}

When working over the field $\mathbb C$, the Specht modules are irreducible, and they form a complete set of irreducible representations of the symmetric group. Polytabloids associated with the standard Young tableaux form a basis for the Specht modules and hence, the Specht modules can be indexed by partitions. Given $\mu \vdash n$, let $\mathbb S^\mu$ denote the Specht module of the symmetric group $S_n$, indexed by partition $\mu$ (see e.g. \cite{books/daglib/0077285} for more details on the construction of Specht modules).

The Kronecker coefficients $g(\mu,\nu, \lambda)$ are defined as the multiplicity of $\mathbb S^{\lambda}$ in the tensor product decomposition of $\mathbb S^{\mu} \otimes \mathbb S^{\nu}$. In particular, for any $\mu,\nu,\lambda \vdash n$,  we can write 
\[\mathbb S^{\mu} \otimes \mathbb S^{\nu} = \oplus_{\lambda \vdash n} \mathbb S^{\oplus g(\mu,\nu, \lambda)}.\] 
We can also write \[\chi^\mu  \chi^\nu = \sum_{\lambda \vdash n} g(\mu,\nu, \lambda) \chi^{\lambda},\] and it follows that \[g(\mu,\nu, \lambda) = \langle \chi^{\mu} \chi^{\nu}, \chi^{\lambda}\rangle = \frac{1}{n!}\sum_{\omega \in \mathfrak S_n} \chi^{\mu}[\omega]\chi^{\nu}[\omega]\chi^{\lambda}[\omega].\] It follows that the Kronecker coefficients have full symmetry over its three parameters $\mu, \nu, \lambda \vdash n$. Further, since $1^n$ is the sign representation, we have $\chi^{\mu} \chi^{1^n} = \chi^{\mu'}$ and therefore the Kronecker coefficients have the transposition property, namely \[g(\mu,\nu,\lambda) = g(\mu',\nu',\lambda) = g(\mu',\nu, \lambda') = g(\mu,\nu',\lambda').\]

\subsection{Symmetric functions}
For main definitions and properties of symmetric functions, we refer to \cite{stanley1997enumerative} Chapter 7.
Let $h_\lambda$ denote the homogeneous symmetric functions and $s_\lambda$ denote the Schur functions. The Jacobi-Trudi Identity (see e.g. \cite{stanley1997enumerative}) is a powerful tool in our work:
\begin{thm}[Jacobi-Trudi Indentity]\label{jacobi} Let $\lambda = (\lambda_1,\dots,\lambda_n)$. Then \[s_\lambda = \det(h_{\lambda_i+j-i}) _{1\le i,j,\le n}\text{ and }s_{\lambda'} = \det(e_{\lambda_i+j-i}) _{1\le i,j,\le n}.\]
\end{thm}
Let $c_{\mu\nu}^{\lambda}$, where $|\lambda| = |\mu|+|\nu|$, denote the Littlewood-Richardson coefficients. Using the Hall inner product on symmetric functions, one can define the Littlewood-Richardson coefficients as $$c_{\mu\nu}^{\lambda} = \langle s_\lambda,s_\mu s_\nu\rangle = \langle s_{\lambda / \mu}, s_\nu\rangle.$$ Namely, the Littlewood-Richardson coefficients are defined to be the multiplicity of $s_\lambda$ in the decomposition of $s_\mu\cdot s_\nu.$ It is well-known that the Littlewood-Richardson coefficients have a combinatorial interpretation in terms of certain semistandard Young tableaux (see e.g.  \cite{stanley1997enumerative}, \cite{books/daglib/0077285}). 

Using the Frobenius map, one can define the Kronecker product of symmetric functions as \[s_\mu \ast s_\nu = \sum_{\lambda \vdash n} g(\mu,\nu,\lambda) s_\lambda.\]

In \cite{Littlewood1958}, Littlewood proved the following identity, which is used frequently in our calculations:
\begin{thm}[Littlewood's Identity]\label{littlewood} Let $\mu,\nu,\lambda$ be partitions. Then
\[s_\mu s_\nu \ast s_\lambda = \sum_{\gamma \vdash |\mu|}\sum_{\delta \vdash |\nu|} c_{\gamma\delta}^{\lambda} (s_{\mu}\ast s_{\gamma})(s_{\nu}\ast s_{\delta}),\]
where $ c_{\gamma\delta}^{\lambda}$ is the Littlewood-Richardson coefficient.
\end{thm}

Another useful tool to simplify our calculations is Pieri's rule:
\begin{thm}[Pieri's rule]\label{Pieri's rule} Let $\mu$ be a partition. Then \[s_\mu s_{(n)} = \sum_\lambda s_{\lambda},\]
summed is over all partitions $\lambda$ obtained from $\mu$ by adding $n$ boxes, with no two added elements in the same column.
\end{thm}

\subsection{Semigroup property}
Semigroup property, which was proved in \cite{Christandl2007}, has been used extensively to prove the positivity of some families of partitions. 

For two partitions $\lambda = (\lambda_1,\lambda_2,\dots \lambda_k)$ and $\mu =(\mu_1,\mu_2,\dots \mu_l)$ with $k \le l$, the horizontal sum of $\lambda$ and $\mu$ is defined as $\lambda+_H \mu = \mu+_H \lambda = (\lambda_1+\mu_1, \lambda_2+\mu_2,\dots, \lambda_k+\mu_k, \mu_{k+1},\dots,\mu_l)$. The vertical sum of two partitions can be defined analogously, by adding the column lengths instead of row lengths. We define the vertical sum $\lambda +_{V} \mu$ of two partitions $\lambda$ and $\mu$ to be $(\lambda' +_H \mu')'$. 
\begin{thm}[Semigroup Property \cite{Christandl2007}] \label{semigroup} If $g(\lambda^1,\lambda^2,\lambda^3)>0$ and $g(\mu^1,\mu^2,\mu^3)>0$, then \\$g(\lambda^1 +_H \mu^1, \lambda^2 +_H \mu^2, \lambda^3 +_H \mu^3) > 0.$
\end{thm}

\begin{cor} If $g(\lambda^1,\lambda^2,\lambda^3)>0$ and $g(\mu^1,\mu^2,\mu^3)>0$, then $g(\lambda^1 +_V \mu^1, \lambda^2 +_V \mu^2, \lambda^3 +_H \mu^3) > 0.$
\end{cor}
Note that by induction, we can extend the semigroup property to an arbitrary number of partitions and a modified version of the semigroup property allows us to use an even number of vertical additions.



\section{Missing partitions in tensor squares of square partitions}\label{section:missing}
In this section, we will show the absence of partitions in the tensor squares of square partitions by discussing the occurrences of two special families of partitions. Note that it follows immediately that the square shape partitions does not satisfy the Tensor Square Conjecture.

\subsection{Near two-row partitions \texorpdfstring{$(m^2-k,k-1,1)$}{}}
Recall that we let $P_k(m)$ denote the number of partitions of $k$ that fit into an $m \times m$ square and let $n = m^2$. The following lemma is proved in \cite{Pak_2013}, see also \cite{VALLEJO2014243}.
\begin{lemma}[\cite{Pak_2013},\cite{VALLEJO2014243}]\label{numofpart} For $1 \le k \le n$, $g(\msquare_m,\msquare_m,(n-k,k)) = P_k(m) - P_{k-1}(m)$.
\end{lemma}

Let $\lambda^\ast$ denotes the $m^n$-complement of $\lambda$ with $m = \lambda_1$ and $n = \lambda_1'.$ We define a $\pi$-rotation of a partition $\lambda$ is the shape obtained by rotating $\lambda$ by $180 ^{\circ}$.
Following Thomas and Yong (\cite{thomas_yong_2010}), let the $m^n$-shortness of $\lambda$ denote the length of the shortest straight line segment of the path of length $m+n$ from the southwest to the northeast corner of $m\times n$ rectangle that separates $\lambda$ from the $\pi$-rotation of $\lambda^\ast.$ 
\begin{ex} Consider $\lambda = (8,4,2,2,1) $, $m = \lambda_1 = 8$ and $n =\lambda_1' = 5$. Then $\lambda^\ast = (7,6,6,4).$ The diagram below is a demonstration for the path of length $m+n$ from the southwest to the northeast corner of a $8\times 5$ rectangle that separates $(8,4,2,2,1)$ from the $\pi$-rotation of $(7,6,6,4).$ The shortest straight line segment of the blue path is $1$. Therefore, the $8^5$-shortness of $(8,4,2,2,1)$ is $1$.

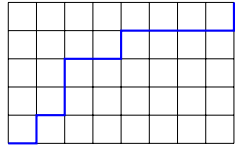
\begin{center}
\begin{tikzpicture}[scale=0.75]
\draw[step=0.5cm,black,very thin] (0,0) grid (4,2.5);
\draw[blue, thick] (0,0) -- (0.5,0) --(0.5,0.5) -- (0.5,0.5) --(1,0.5)--(1,1.5) -- (2,1.5)--(2,2)--(4,2) --(4,2.5) ;
\end{tikzpicture}
\end{center}
\end{ex}
\begin{ex} 
Now consider $\lambda = (8,8,8,3,3) $, $m = \lambda_1 = 8$ and $n =\lambda_1' = 5$. Then $\lambda^\ast = (5,5).$ From the diagram below, we can see the lengths of straight line segments of the blue path are $2,2,6,3$, and hence the shortest straight line segment of the blue path is $2$. Therefore, the $8^5$-shortness of $(8,8,8,3,3)$ is $2$.

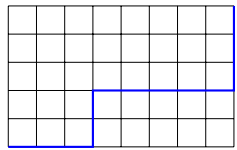
\begin{center}
\begin{tikzpicture}[scale=0.75]
\draw[step=0.5cm,black,very thin] (0,0) grid (4,2.5);
\draw[blue, thick] (0,0) -- (1.5,0) --(1.5,1) -- (4,1) --(4,2.5) ;
\end{tikzpicture}
\end{center}
\end{ex}

For the following theorem, jointly due to Gutschwager, Thomas and Yong, we follow \cite{dou2009hive}:
\begin{thm}[\cite{Gutschwager_2010}, \cite{thomas_yong_2010}]\label{ty}
The basic skew Schur function $s_{\lambda/\mu}$ is multiplicity-free if and only if at least one of the following is true:
\begin{enumerate}[(i)]
    \item $\mu$ or $\lambda^\ast$ is the zero partition 0;
    \item $\mu$ or $\lambda^\ast$ is a rectangle of $m^n$-shortness 1;
    \item $\mu$ is a rectangle of $m^n$-shortness 2 and $\lambda^\ast$ is a fat hook (or vice versa);
    \item $\mu$ is a rectangle and $\lambda^\ast$ is a fat hook of $m^n$-shortness 1 (or vice versa);
    \item $\mu$ and $\lambda^\ast$ are rectangles;
\end{enumerate}
where $\lambda^\ast$ denotes the $m^n$-complement of $\lambda$ with $m = \lambda_1$ and $n = \lambda_1'.$
\end{thm}

\begin{cor} \label{mult-free} Let $\lambda_m = (m^{m-1}, m-1)$ denote the chopped square partition of size $m^2-1.$ For every pair of partitions $\beta$ and $\mu$ such that $|\beta| + |\mu| = m^2-1$, $c_{\beta\mu}^{\lambda_m} \in\{0,1\}.$ 
\end{cor}

\begin{proof} Let $\lambda_m^\ast$ denote the $m^m$-complement of $\lambda_m$. Then $\lambda_m^\ast = (1).$ The lengths of straight line segments of the path from the southwest to the northeast corner that separates $\lambda_m$ from $\lambda_m^\ast$ are $m-1,1,1, m-1$, and therefore the $m^m$-shortness of $\lambda_m^\ast$ is 1. Let $\beta \vdash k \le m^2-1$. Then, $s_{\lambda_m/\beta}$ is a basic skew Schur function as the difference between consecutive rows in $\lambda_m$ is at most 1. By Theorem \ref{ty} A1, $s_{\lambda_m/\beta}$ is multiplicity-free, which implies that $c_{\beta\mu}^{\lambda_m} \in \{0,1\}$ for any $\mu \vdash m^2-1-k$.
\end{proof}
 
\begin{lemma} Let $\ell_1(\alpha)$ denote the number of different parts of partition $\alpha$. For $1 \le k \le m$,
\[\sum_{\beta \vdash k-1}\sum_{\mu \vdash m^2-k} c_{\beta\mu}^{\lambda_m}= P_{k-1}(m) + \sum_{\beta \vdash k-1} \ell_1(\beta).\] 
\end{lemma}
\begin{proof} Let $1 \le k \le m$. Since $c_{\beta\mu}^{\lambda_m} = 1 > 0 $, partitions $\beta, \mu \subseteq \lambda_m \subseteq \msquare_m$. Let $\beta^\ast$ and $\lambda_m^\ast$ denote the complements of $\beta$ and $\lambda_m$ inside the $m \times m$ square, respectively. Since $c_{\beta\mu}^{\lambda_m}$ depends only on $\mu$ and the skew partition $\lambda_m / \beta$, and the skew partitions $\lambda_m / \beta$ and $\beta ^\ast / \lambda_m^\ast$ are identical when rotated, we have $c_{\beta\mu}^{\lambda_m}  = c_{\lambda_m^{\ast}\mu}^{\beta^{\ast}} =c_{(1)\mu}^{\beta^{\ast}}$.
By the Pieri's rule (Theorem \ref{Pieri's rule}), $c_{(1)\mu}^{\beta^{\ast}} =1 $ if and only if  $\mu$ is a partition obtained from $\beta^{\ast}$ by removing 1 element. Since the number of ways to obtain a partition by removing an element from $\beta^{\ast}$ is $\ell_1(\beta^{\ast})$, we have \[\sum_{\beta \vdash k-1}\sum_{\mu \vdash m^2-k} c_{\beta\mu}^{\lambda_m}=\sum_{\substack{\beta \vdash k-1\\ \beta \subseteq \lambda_m}}\sum_{\mu \vdash m^2-k} c_{(1)\mu}^{\beta^{\ast}}  =\sum_{\substack{\beta \vdash k-1\\ \beta \subseteq \lambda_m}} \ell_1(\beta^{\ast}).\] 
Note that $\ell_1(\beta^{\ast}) = \ell_1(\beta) - 1$ if $\beta_1 = \ell(\beta) = m$; $\ell_1(\beta^{\ast}) = \ell_1(\beta)$ if exactly one of $\beta_1, \ell(\beta)$ is $m$;   
otherwise, $\ell_1(\beta^{\ast}) = \ell_1(\beta)+1$. Hence, when $1 \le k \le m$, we have $$\sum_{\substack{\beta \vdash k-1\\ \beta \subseteq \lambda_m}} \ell_1(\beta^{\ast}) =  P_{k-1}(m) + \sum_{\beta \vdash k-1} \ell_1(\beta).$$\end{proof}

\begin{prop}[near two-row partitions]\label{prop:zerocase} Let $2\le k \le m$. Let $f(k)$ denote the number of partitions of $k$ with no parts equal to 1 or 2, and $\ell_1(\alpha)$ denote the number of different parts of partition $\alpha$. Then \[g(\msquare_m,\msquare_m,(n-k,k-1,1)) = \sum_{\substack{\alpha \vdash k-1\\ \alpha_1=\alpha_2}} \ell_1(\alpha) - f(k).\]
\end{prop}
\begin{proof} Letting $s_\lambda$ denote the Schur function indexed by a partition $\lambda$, we have $$g(\msquare_m,\msquare_m,(n-k,k-1,1)) = \left\langle s_{\msquare_m}, s_{(n-k,k-1,1)}\ast s_{\msquare_m} \right\rangle.$$
Observe that, by Pieri's rule (Theorem \ref{Pieri's rule}), we have \[s_{(n-k,k-1)}s_{(1)} =s_{(n-k,k-1,1)}+s_{(n-k+1,k-1)} +s_{(n-k,k)}.\]
Rewriting the above identity gives us that $g(\msquare_m,\msquare_m,(n-k,k-1,1))$ can be interpreted as \[\left \langle s_{\msquare_m}, (s_{(n-k,k-1)}s_{(1)})\ast s_{\msquare_m} \right \rangle - \left \langle s_{\msquare_m}, s_{(n-k+1,k-1)}\ast s_{\msquare_m} \right \rangle - \left \langle s_{\msquare_m}, s_{(n-k,k)}\ast s_{\msquare_m} \right \rangle.\]
We first note that the last two terms give two Kronecker coefficients $g(\msquare_m,\msquare_m,(n-k+1,k-1))$ and $g(\msquare_m,\msquare_m,(n-k,k))$. Notice that by Lemma \ref{numofpart}, we have 
$$g(\msquare_m,\msquare_m,(n-k+1,k-1))=P_{k-1}(m)-P_{k-2}(m)$$ and 
$$g(\msquare_m,\msquare_m,(n-k,k))=P_k(m)- P_{k-1}(m).$$ By Littlewood's Identity (Theorem \ref{littlewood}), \begin{align*}
    (s_{(n-k,k-1)}s_{(1)})\ast s_{\msquare_m} &= \sum_{\gamma \vdash {n-1}} c_{\gamma, (1)}^{\msquare_m}(s_{(n-k,k-1)}\ast s_\gamma)(s_{(1)}\ast s_{(1)})\\
    &= (s_{(n-k,k-1)}\ast s_{\lambda_m})(s_{(1)}),
\end{align*}
as $c_{\gamma,(1)}^{\msquare_m} =1$ if $\gamma = \lambda_m$ and $c_{\gamma,(1)}^{\msquare_m} =0$ for all the other partitions of size $n-1$. Taking inner product with $s_{\msquare_m}$ on both sides, we have 
\begin{align*} \left \langle s_{\msquare_m}, (s_{(n-k,k-1)}s_{(1)})\ast s_{\msquare_m} \right \rangle &= \left \langle s_{\msquare_m},(s_{(n-k,k-1)}\ast s_{\lambda_m})(s_{(1)}) \right \rangle\\
&= \left \langle s_{\msquare_m \setminus (1)},(s_{(n-k,k-1)}\ast s_{\lambda_m}) \right \rangle\\
&= \langle s_{\lambda_m}, s_{(n-k,k-1)}\ast s_{\lambda_m}\rangle.
\end{align*}
By Littlewood's Identity (\ref{littlewood}), Jacobi-Trudi Identity (\ref{jacobi}), together with Corollary \ref{mult-free}, $c_{\mu\beta}^{\lambda_m} \in \{0,1\}$, we have
\begin{align*}
    \left \langle s_{\lambda_m}, s_{(n-k,k-1)}\ast s_{\lambda_m} \right \rangle &= \sum_{\beta \vdash k-1}\sum_{\mu \vdash n-k} (c_{\mu\beta}^{\lambda_m})^2-\sum_{\alpha \vdash k-2}\sum_{\gamma \vdash n-k+1} (c_{\alpha\gamma}^{\lambda_m})^2 \\
    &= \sum_{\beta \vdash k-1}\sum_{\mu \vdash n-k} c_{\mu\beta}^{\lambda_m}-\sum_{\alpha \vdash k-2}\sum_{\gamma \vdash n-k+1} c_{\alpha\gamma}^{\lambda_m}.
\end{align*}
Putting the pieces together, we then have 
\begin{align*}
 &g(\msquare_m,\msquare_m,(n-k,k-1,1))\\
 =\,\,& \left \langle s_{\lambda_m}, s_{(n-k,k-1)}\ast s_{\lambda_m} \right \rangle - g(\msquare_m,\msquare_m,(n-k+1,k-1)) - g(\msquare_m,\msquare_m,(n-k,k))\\
 =\,\,& \langle s_{\lambda_m}, s_{(n-k,k-1)}\ast s_{\lambda_m}\rangle -(P_{k-1}(m)-P_{k-2}(m))-(P_k(m)- P_{k-1}(m))\\
 =\,\,&\sum_{\beta \vdash k-1}\sum_{\mu \vdash n-k} c_{\mu\beta}^{\lambda_m}-\sum_{\alpha \vdash k-2}\sum_{\gamma \vdash n-k+1} c_{\alpha\gamma}^{\lambda_m}-(P_k(m)- P_{k-2}(m))\\
 =\,\,& P_{k-1}(m)+\sum_{\beta \vdash k-1 
 }\ell_1(\beta)-\left(P_{k-2}(m)+ \sum_{\alpha \vdash k-2 }\ell_1(\alpha)\right)-(P_k(m)- P_{k-2}(m))\\
 =\,\,&\sum_{\beta \vdash k-1}\ell_1(\beta)-\sum_{\alpha \vdash k-2 }\ell_1(\alpha)-(P_k(m)-P_{k-1}(m))\\
 =\,\,&\sum_{\substack{\beta \vdash k-1\\ \beta_1=\beta_2}}\ell_1(\beta) +\left(\sum_{\substack{\beta \vdash k-1\\ \beta_1>\beta_2}}\ell_1(\beta) -\sum_{\alpha \vdash k-2 }\ell_1(\alpha)\right)-(P_k(m)-P_{k-1}(m))\\
 =\,\,&\sum_{\substack{\beta \vdash k-1\\ \beta_1=\beta_2}}\ell_1(\beta) + \sum_{\substack{\beta \vdash k-2\\ \beta_1 = \beta_2}} 1 -(P_k(m)-P_{k-1}(m))\\
  =\,\,&\sum_{\substack{\beta \vdash k-1\\ \beta_1=\beta_2}}\ell_1(\beta) -\left(P_k(m)-P_{k-1}(m) - \sum_{\substack{\beta \vdash k-2\\ \beta_1 = \beta_2}} 1 \right)\\
 =\,\,&\sum_{\substack{\alpha \vdash k-1\\ \alpha_1=\alpha_2}}  \ell_1(\alpha) - f(k).\end{align*}

\end{proof}

The following result, which provides a necessary and sufficient condition for a near two-row partition with a short second row to vanish in the tensor square of square partitions, follows from Theorem \ref{prop:zerocase}.
\begin{cor} \label{cor:zerocase} Let $2\le k \le m$. Then $g(\msquare_m,\msquare_m,(n-k,k-1,1)) = 0$ if and only if  $k \le 4$. \end{cor}
\begin{proof}
We can easily verify that $\sum_{\substack{\alpha \vdash k-1\\ \alpha_1=\alpha_2}}  \ell_1(\alpha) = f(k)$ for $k \in \{2,3,4\}$. Then by Proposition \ref{prop:zerocase}, we conclude that $g(\msquare_m,\msquare_m,(n-k,k-1,1)) = 0$ when $k \le 4.$

Next, we consider the case when $k \ge 5$. We can establish an injection from the set of all partitions of $k$ whose parts are at least $3$ to the set of partitions of $k-1$ whose first two parts are the same, that is from \[S = \{\beta \vdash k \mid \beta_i \notin \{1,2\} \text{ for all } i\}\] to \[T = \{\alpha \vdash k-1 \mid \alpha_1 = \alpha_2 \}.\] This injection is achieved by removing one box from the last row of $\beta \in S$ and taking the transpose. When $k\ge 5$, it follows that $\sum_{\substack{\alpha \vdash k-1\\ \alpha_1=\alpha_2}}  \ell_1(\alpha) > |T| \ge |S| = f(k)$. Hence, we conclude that $g(\msquare_m,\msquare_m,(n-k,k-1,1)) > 0$.
\end{proof}

\subsection{Hooks}
The following results on hook positivity are due to Ikenmeyer and Panova:
\begin{thm}[\cite{ikenmeyer2017rectangular}]\label{thm:hookpos} Let $b \ge 7.$ Assume that $m \ge b$.  We have $g((mb-k,1^k),b\times m, b\times m) > 0$ for $k \in [0,b^2-1]\setminus \{1,2,4,6,b^2-2,b^2-3,b
^2-5,b^2-7\}$ and is 0 for all other values of $k.$
\end{thm}
By Theorem \ref{thm:hookpos} and results in the previous section, we prove the forward direction of Conjecture \ref{conj:myconj}:
\begin{cor} \label{zerocor} For $m \ge 7$, $ g(\msquare_m, \msquare_m,\mu) = 0$ if $\mu \in S$ or $\mu' \in S$, where \[S = \{(m^2-3,2,1), (m^2-4,3,1), (m^2-j,1^j)\mid j \in \{1,2,4,6\}\}.\]\end{cor}
\begin{proof}
It follows directly from Theorem \ref{thm:hookpos} and Corollary \ref{cor:zerocase}.
\end{proof}
\section{Constituency of families of partitions of special shapes}\label{sec:constituency}
In this section, we will discuss the constituency of three families of special shapes in tensor squares of square partitions, including two-row partitions, near two-row partitions, and three-row partitions.

\subsection{Two-row partitions}
The following Theorem shown in \cite{pak14unimodality} is a generalization of Lemma \ref{numofpart} and it tells us how to compute the Kronecker coefficients of the form $g(m^l, m^l, (lm-k,k))$.
\begin{thm}[\cite{pak14unimodality}] \label{t:strict}

Let $n = lm$, $\tau_k = (n-k,k)$, where $0 \le k \le n/2$ and set $p_{-1}(l,m) = 0$. Then \[g(m^l, m^l, \tau_k) = p_k(l,m) - p_{k-1}(l,m).\]Furthermore, when $l,m \ge 8$, $g(m^l, m^l, \tau_k)  > 0$ when $k \ge 2.$
\end{thm}

\begin{cor}\label{cor:two-row square}
Let $m \ge 7$. For any $1 \le k \le m^2-2$,  $g(\msquare_m,\msquare_m,(m^2-k,k))>0$.
\end{cor}
\begin{proof}
By direct computation using the formula in Theorem \ref{t:strict}, we can verify that the statement holds for $m = 7$. By strict unimodality of $q$-binomial coefficients as shown in \cite{Pak_2013}, we can obtain positivity of the Kronecker coefficients of the form $g(\msquare_m,\msquare_m,(m^2-k,k))$ for every $m\ge 8.$ \end{proof}

\subsection{Near two-row partitions}

We will first consider the occurrences of near two-row partitions $(m^2-k,k-1,1)$ with a second row longer than $m-1$. The following is one of our main results and is proven by considering different cases depending on different values of $k$ and the parity of $m$.
\begin{thm}\label{near two-row}
Let $m$ be an integer. For every $m\ge 5$, $g(\msquare_m,\msquare_m,(m^2-k,k-1,1))> 0$ if and only if  $k\ge 5.$
\end{thm}

The following is a well-known result on tensor square of $2\times n$ rectangles from \cite{10.36045/bbms/1103408635}:
\begin{thm}[\cite{10.36045/bbms/1103408635}]\label{thm:nn}
The Kronecker coefficient $g((n,n),(n,n),\mu)>0$ if and only if either $\ell(\mu) \le 4$ and all parts even or $\ell(\mu) = 4$ and all parts odd.
\end{thm}

When $m$ is even and $5 \le k \le \frac{m^2}{2}$, we decompose $\msquare_{m}$ as $\msquare_{m} = (\msquare_{m-2} +_V (m-2, m-2)) +_H (2^{m})$. We can find a horizontal decomposition $(m^2-k,k-1,1) = \mu^1+_H \mu^2 +_H\mu^3$ where $\mu^1$ is a three-row partition with the second row longer than 4 and the third row equal to 1, and $\mu^2$ and $\mu^3$ are partitions of $2m$ and $2m-4$ with all parts even. Then by induction and semigroup property, we have:
\begin{prop} \label{prop:even} For every even number $m \ge 6$, $g(\msquare_m,\msquare_m,(m^2-k,k-1,1))> 0$ for every $m+1 \le k \le \frac{m^2}{2}$.
\end{prop}

\begin{proof}
For an even integer $m\ge 6$, we can write $m = 2r$ where $r\ge 3$. We shall proceed by induction on $r$. Based on computational evidence, we observe that $g(\msquare_6,\msquare_6,(6^2-k,k-1,1))> 0$ for every $7 \le k \le 18$. 

Let $r \ge 4$. Assume the inductive hypothesis that $g(\msquare_{2(r-1)},\msquare_{2(r-1)},(4(r-1)^2-i,i-1,1))$ for any $2r-1 \le i \le 2(r-1)^2$.
Let $2(r+1)\le k \le 2r^2.$ We can decompose the square partition with side length $2r$ as follows:  \[\msquare_{2r} = \left (\msquare_{2(r-1)} +_V (2r-2, 2r-2)\right ) +_H (2^{2r}).\]

Note that by Theorem \ref{thm:nn} and the transposition property of Kronecker coefficients, we obtain that $g((2^{2r}),(2^{2r}), (2(r+a),2(r-a))) > 0 $ for any $0\le a\le r$,  and $g((2r-2,2r-2),(2r-2,2r-2), (2(r-1+b),2(r-1-b)))>0$ for any $0\le b\le r-1.$

Consider the following system of inequalities:
\[
\begin{cases}
4r^2-k-2(r+a) \ge k-1-2(r-a)\\
4r^2-k-2(r+a)-2(r-1+b) \ge k-1-2(r-a)-2(r-1-b) \ge 1\\
k-1-2(r-a)-2(r-1-b) \ge 5
\end{cases}.
\]
Suppose that $0\le a\le r$, $0\le b\le r-1$ is a pair of solutions to the system. We define partition $\alpha(a,b):=(4r^2-k-2(r+a)-2(r-1+b), k-1-2(r-a)-2(r-1-b),1)$. By inductive hypothesis, together with Corollary \ref{cor:zerocase}, $g(\msquare_{2(r-1)},\msquare_{2(r-1)}, \alpha(a,b)) > 0.$
Note that we can decompose the near two-row partition as \[(4r^2-k,k-1,1) = \alpha(a,b) +_H(2(r-1+b),2(r-1-b)+_H (2(r+a),2(r-a)).\] Then by semigroup property (Theorem \ref{semigroup}), $g\left (\msquare_{2r},\msquare_{2r},(4r^2-k,k-1,1)\right)> 0$. By the Principle of Mathematical Induction, the statement holds for all even $m \ge 6.$ 

Hence, it suffices to show the system of inequalities has integral solutions $0\le a\le r,0\le b\le r-1$. By simplifying and rearranging, we can further reduce this system of inequalities to:
\[
\begin{cases} 
a \le r^2-\frac{k}{2}+\frac{1}{4}\\
2r+2-\frac{k}{2} \le a+b \le r^2-\frac{k}{2}+\frac{1}{4}.
\end{cases}
\]
Notice that when $k \le \frac{(2r-1)^2}{2}$, the values $a = r$ and $b = \max\{ \lceil r+2-\frac{k}{2}\rceil, 0\}$ provide a feasible solution to the system. When $\frac{(2r-1)^2}{2} \le k \le 2r^2$, the values $a = \lfloor r^2-\frac{k}{2}+\frac{1}{4}\rfloor$ and $b = 0$ provide a feasible solution to the system.
\end{proof}

\begin{ex} Let $m = 6$ and $k = 10$. Diagrams below illustrate a way to decompose partitions $\msquare_6$ and $(26,9,1)$. Since $g(\msquare_4, \msquare_4, (8,7,1))> 0$, $g(2^6, 2^6, (10,2))>0$ by Theorem \ref{thm:nn} and $g((4,4),(4,4),(8))>0$, we conclude that $g(\msquare_6,\msquare_6,(26,9,1))> 0$ by semigroup property.
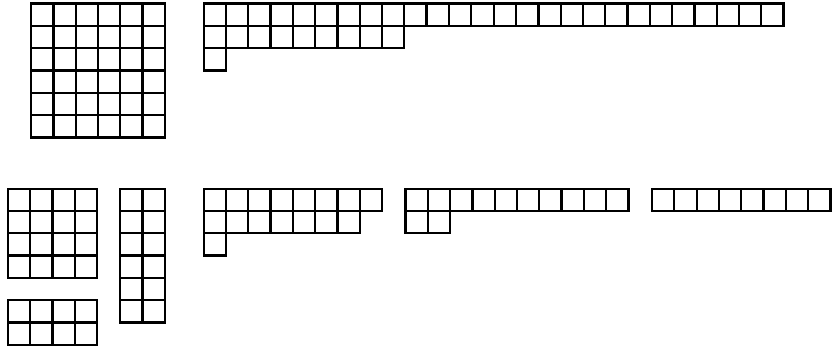
\ytableausetup{boxsize = 8px}
\begin{align*}
\\
\ydiagram{6,6,6,6,6,6} \hspace{.5cm} &\ydiagram{26,9,1} \\
\\
\ydiagram{4,4,4,4,0,4,4} \ydiagram{1+2,1+2,1+2,1+2,1+2,1+2} \hspace{.5cm}&\ydiagram{8,7,1} \ydiagram{1+10,1+2}\ydiagram{1+8}
\end{align*}
\end{ex}

We will next prove the positivity of $g(\msquare_m, \msquare_m, (m^2-k,k-1,1))$ when $m$ is odd using the semigroup property. 
\begin{prop} \label{prop:short2ndrow} For every odd integer $m \ge 7$ and $k\ge 5$ such that $k\le \frac{(m-1)^2+1}{2}$, $g(\msquare_m,\msquare_m,(m^2-k,k-1,1))> 0$.
\end{prop}
\begin{proof}
Let $m\ge 7$ and $k \ge 5$. Note that when $k\le \frac{(m-1)^2+1}{2}$, we have $(m^2-k)-(k-1) \ge 2m-1$ and we can consider the decompositions \[(m^2-k, k-1,1) = (m^2-k-2m+1, k-1, 1) +_H(m-1) +_H (m)\] and \[\msquare_m = (\msquare_{m-1}+_V(m-1) )+_H (1^m).\] Then by semigroup property and Proposition \ref{prop:even}, we have $g(\msquare_m,\msquare_m,(m^2-k,k-1,1))> 0$ in this case.
\end{proof}

Note that the previous proof only establishes the constituency of near two-row partitions with a relatively short second row in the tensor square of square partitions with an odd side length. Now we aim to demonstrate the constituency of near two-row partitions whose first part and second part have similar sizes. To accomplish this, we will first establish the constituency of an extreme case where the second row has a maximal length:

\begin{lemma}\label{lem:kk1}
For every odd integer $m \ge 3$, $g\left(\msquare_m,\msquare_m,\left(\frac{m^2-1}{2},\frac{m^2-1}{2},1\right)\right)> 0.$
\end{lemma}

\begin{proof}  We can write odd integers $m$ as $ m = 2k+1$, and we will proceed with a proof by induction on $k \ge 1$. 

We can verify the statement directly for $m \in \{3,5,7\}$ through direct computations. 
When $k = 4$, we have $m = 2k+1 = 9$. In this case, the square partition $\msquare_{9}$ can be expressed as  \[\msquare_{9} = ((5^5) +_{V} (5^4)) +_H (4^9).\] Furthermore, we can write \[(40,40,1) = (12,12,1) +_H (10,10) +_H (18,18).\] By assumption, we have $g((5^5), (5^5),(12,12,1))> 0$. Using computer software, we can verify the positivity of $g((5^4), (5^4),(10,10))$ and $g((4^9), (4^9),(18,18))$. Therefore, by the semigroup property, we conclude that $g\left(\msquare_9,\msquare_9,(40,40,1)\right)> 0$.

Now let $k \geq 5$ and $m = 2k+1$. By the inductive hypothesis, we assume that \[g\left(\msquare_{m'},\msquare_{m'},\left(\frac{{m'}^2-1}{2},\frac{{m'}^2-1}{2},1\right)\right)> 0,\] holds for all $m' = 2k' + 1 < 2k+1$. We can express $\msquare_{m}$ as \[\msquare_{m} = (((m-8)^{(m-8)}) +_{V} ((m-8)^8)) +_H (8^m).\] Furthermore, we have \[\left(\frac{m^2-1}{2},\frac{m^2-1}{2},1\right) = \left(\frac{(m-8)^2-1}{2}, \frac{(m-8)^2-1}{2}, 1\right) +_H (4(m-8), 4(m-8)) +_H (4m,4m).\]

Using Theorem \ref{t:strict}, we know that $g((8^m), (8^m), (4m,4m)) > 0$. In the case of $k = 5$, where $m = 11$, we can directly compute and show the positivity of $g((m-8)^8, (m-8)^8, (4(m-8),4(m-8)))$. For $k \in \{6,7\}$, we can use the semigroup property and Theorem \ref{thm:nn} to establish the positivity of $g((m-8)^8, (m-8)^8, (4(m-8),4(m-8)))$ since $g((8,8),(8,8),(8,8)) > 0$. For $k \geq 8$, the positivity of $g((m-8)^8, (m-8)^8, (4(m-8),4(m-8)))$ follows from Theorem \ref{t:strict}. Additionally, by the inductive hypothesis, we have $$g\left(\msquare_{m-8}, \msquare_{m-8}, \left(\frac{(m-8)^2-1}{2}, \frac{(m-8)^2-1}{2}, 1\right)\right) > 0.$$ By the semigroup property (\ref{semigroup}), we conclude that
\[g\left(\msquare_m,\msquare_m,\left(\frac{m^2-1}{2},\frac{m^2-1}{2},1\right)\right)> 0,\]
which completes the induction.
\end{proof}

\begin{cor}\label{kk1-rec}
For every pair of odd integers $l,m \ge 11$, $g(m\times l, m\times l, (\frac{ml-1}{2},\frac{ml-1}{2},1))>0.$
\end{cor}
\begin{proof}
By Lemma \ref{lem:kk1}, we know that $g(\msquare_m, \msquare_m, (\frac{m^2-1}{2},\frac{m^2-1}{2},1))>0$ for any odd integer $m \ge 3.$

Let $m,l$ be odd integers. Without loss of generality, assume that $m \ge l$. If $|m-l| \equiv 0 \mod 4$, then we can write the square partition of shape $m \times l$ as $\msquare_l +_V (l^{(m-l)})$. Since $m-l$ is a multiple of $4$, by Lemma \ref{multipleof4} and the semigroup property, we conclude that $g(m\times l, m\times l, (\frac{ml-1}{2},\frac{ml-1}{2},1))>0.$ If $|m-l| \equiv 2 \mod 4$, we can write $m \times l $ as $m \times l = 10 \times l +_V (m-10)\times l$. Note then $(m-10-l) \equiv 0 \mod 4$, and by Theorem \ref{t:strict}, $g(10\times l, 10\times l, (5l,5l))>0.$ Hence, by semigroup property, we conclude that $g(m\times l, m\times l, (\frac{ml-1}{2},\frac{ml-1}{2},1))>0$ for any odd integers $m,l \ge 11.$
\end{proof}




\begin{lemma}\label{multipleof4}
For every integer $m\ge 2$, $g(m^4,m^4,(2m,2m))>0.$
\end{lemma}
\begin{proof}
If $m$ is even, it follows from Theorem \ref{thm:nn}.
If $m$ is odd, we first note that with the help of the computer, one can check that $g(3^4,3^4,(6,6))>0$. Then we can decompose the partition $m^4$ as $m^4 = 3^4 +_H (m-3)^4$. Since $m-3$ is even, we have $g((m-3)^4, (m-3)^4, (2m-6.2m-6))>0$. By semigroup property, we can conclude that $g(m^4,m^4,(2m,2m))>0.$ \end{proof}

\begin{lemma}\label{lem:almostkk1}
For every odd integer $m \ge 3$, $g\left(\msquare_m,\msquare_m,\left(\frac{m^2+1}{2},\frac{m^2-3}{2},1\right)\right)> 0.$
\end{lemma}
\begin{proof}
We can check by direct computation that the statement holds for $m = 3$ and $m = 5.$ Let $m \ge 7$. Suppose that the statement holds for odd numbers less than $m.$ Consider the decomposition $\msquare_m = (\msquare_{m-4} +_V (m-4)^4) +_H (4^{m-4})$. Since $g(4^{m-4}, 4^{m-4}, (2m,2m)) >0$ and $g((m-4)^4, (m-4)^4, (2m-8,2m-8))>0$ by Lemma \ref{multipleof4}, by semigroup property, we have that $g\left(\msquare_m,\msquare_m,\left(\frac{m^2+1}{2},\frac{m^2-3}{2},1\right)\right)> 0$. By induction, $g\left(\msquare_m,\msquare_m,\left(\frac{m^2+1}{2},\frac{m^2-3}{2},1\right)\right)> 0$ for any odd integer $m \ge 3$.
\end{proof}

We will use Lemma \ref{lem:kk1} and Lemma \ref{lem:almostkk1} as ingredients to establish the positivity in the case where $m$ is an odd integer and the first part and second part of the near two-row partition are of similar sizes. 
\begin{prop} \label{odd case}
For every odd integer $m \ge 7$ and $k \ge 5$ such that $k \ge \frac{(m-1)^2}{2}+1$, $g(\msquare_m,\msquare_m,(m^2-k,k-1,1))> 0$. \end{prop}
\begin{proof} We shall prove the statement by induction on odd integers $m \ge 7$. Note that we can check by semigroup property and computer that the statement holds for $m = 7$. 
Let $m \ge 9$ be an odd integer. Suppose that the statement holds for $m -2$. Consider the decomposition that $\msquare_m = (\msquare_{m-2} +_V (m-2,m-2)) +_H (2^m)$. Let $a := (m^2-k)-(k-1)$. Since $k \ge \frac{(m-1)^2}{2}+1$, we have $(m^2-k)-(k-1)\le 2m-2.$ We will discuss three cases as follows.
\begin{enumerate}
    \item [Case 1:] If $a = 0$, by Lemma \ref{lem:kk1}, we know that $g\left(\msquare_m,\msquare_m,\left(\frac{m^2-1}{2},\frac{m^2-1}{2},1\right)\right)> 0.$
    \item [Case 2:] If $a = 2$, by Lemma \ref{lem:almostkk1}, we know that $g\left(\msquare_m,\msquare_m,\left(\frac{m^2+1}{2},\frac{m^2-3}{2},1\right)\right)> 0.$
    \item [Case 3:] If $a > 0$ and $a \equiv 0 \mod 4$, consider the following decomposition of $(m^2-k,k-1,1)$: \[ \left(\frac{(m-2)^2-1}{2},\frac{(m-2)^2-1}{2},1\right)+_H (m+1+ 2x,m-1-2x)+(m-1+2y,m-3-2y),\]
    where $x \le \frac{m-1}{2}, y  \le \frac{m-3}{2}$ are non-negative integers such that $4(x+y+1) = a.$ By Lemma \ref{lem:kk1}, Theorem \ref{thm:nn} and semigroup property, we can conclude that $g(\msquare_m,\msquare_m,(m^2-k,k-1,1))> 0$ in this case.
    \item [Case 4:] If $a >2$ and $a \equiv 2 \mod 4$, consider the following decomposition of $(m^2-k,k-1,1)$: \[ \left(\frac{(m-2)^2+1}{2},\frac{(m-2)^2-3}{2},1\right)+_H (m+1+ 2x,m-1-2x)+(m-1+2y,m-3-2y),\]
    where $x \le \frac{m-1}{2}, y  \le \frac{m-3}{2}$ are non-negative integers such that $4(x+y+1) = a-2.$ By the inductive hypothesis, Theorem \ref{thm:nn} and semigroup property, we can conclude that $g(\msquare_m,\msquare_m,(m^2-k,k-1,1))> 0$ in this case.
\end{enumerate}
\end{proof}
We now put the above pieces together to prove Theorem \ref{near two-row}.
\begin{proof}[Proof of Theorem \ref{near two-row}]
One can check by direct computation that the proposition holds for $m = 5$ and $m = 7$. 
Then the statement follows directly from Corollary \ref{cor:zerocase}, Proposition \ref{prop:even}, Proposition \ref{prop:short2ndrow} and Proposition \ref{odd case}.
\end{proof}

\subsection{Three-row partitions} Next, we consider the case when $\mu$ is a three-row partitions with $\mu_3 \ge 2.$ Below is one of our main results. We will prove it by discussing different cases according to the parity of $m$ and different values of $k$.

\begin{thm}\label{mainthm}
For every odd integer $m\ge 5$, $g(\msquare_m, \msquare_m, \mu)>0$ for any three-row partition $\lambda\vdash m^2$ with $\mu_3 \ge 2$.
\end{thm}

Below are some results that will be used to prove the positivity of $g(\msquare_m, \msquare_m, \mu)$ when $m$ is an even integer.
\begin{prop}\label{rectangular}
Let $l,k,m$ be positive integers such that $lk = m^2$. Then $g(\msquare_m,\msquare_m, k^l) > 0$ if $l \mid k$.
\end{prop}
\begin{proof}
Let $l,k,m$ be positive integers such that $lk = m^2$. Suppose that $l \mid k.$ Then, $l^2 \mid m^2$ and hence $l \mid m$. It follows that we can decompose $\msquare_m$ as \[\msquare_m = \sum_{+_H} \left(\sum_{+_V} \msquare_l\right).\]
By semigroup property, we can conclude that $g(\msquare_m,\msquare_m, k^l) > 0$ 
\end{proof}

\begin{cor}\label{3krectangular}
    If $m$ is a multiple of $3$, then $g\left(\msquare_m,\msquare_m, \left(\frac{m^2}{3},\frac{m^2}{3},\frac{m^2}{3}\right)\right) > 0. $
\end{cor}
\begin{proof}
    It follows from Proposition \ref{rectangular}.
\end{proof}
\begin{lemma}\label{3k}
For $k \ge 3$, $g(3^k, 3^k, k^3) =g(k^3, k^3, k^3) > 0.$
\end{lemma}
\begin{proof}
With the help of the computer, we can check that $g(k^3, k^3, k^3) > 0$ for $k \in \{3,4,5\}.$ For any $k \ge 6$, we can write $k = 3j+r$ for some non-negative integers $j,r$ such that $r \in \{0,4,5\}$. Then, we can write the partition $(k,k,k)$ as a horizontal sum of $j$ square partitions of side length $3$, and the rectangular partition $(r,r,r)$. The generalized semigroup property shows that $g(k^3, k^3, k^3) > 0$ for $k \ge 6$. Furthermore, by the transposition property, we have $g(3^k, 3^k, k^3) =g(k^3, k^3, k^3) > 0$ for $k \ge 3.$
\end{proof}

Lemma \ref{m=2mod3}, \ref{m=1mod3} and \ref{m=0mod3} will be used in the proof of Proposition \ref{even:mu3ge2}. These specific cases are addressed individually due to their different decomposition approach, setting them apart from the remaining cases of the proposition's proof.

\begin{lemma}\label{m=2mod3}
The Kronecker coefficient $g\left(\msquare_m,\msquare_m, \left(\frac{m^2+2}{3},\frac{m^2-1}{3},\frac{m^2-1}{3}\right)\right) > 0$ for any positive integer $m\ge 5$ such that $m \equiv 2 \mod 3$.
\end{lemma}
\begin{proof}
For any positive integer $m\ge 5$ such that $m \equiv 2 \mod 3$, we can write $m = 3r+2$ for some $r \ge 1$. We will prove the proposition by induction on $r$.
When $r = 1$, $3r+2 = 5$ and with the help of the computer, we can check that $g\left(\msquare_{5},\msquare_{5}, \left(\frac{5^2+2}{3},\frac{5^2-1}{3},\frac{5^2-1}{3}\right)\right) > 0.$
Let $r \ge 2$. Assume the statement is true for $r-1$. We can decompose $\msquare_{3r+2}$ as $$\msquare_{3r+2} =\left( \msquare_{3r-1}+_V(3r-1)\right) +_H (3^{3r+2}),$$ and we can decompose the partition $\left(\frac{(3r+2)^2+2}{3},\frac{(3r+2)^2-1}{3},\frac{(3r+2)^2-1}{3}\right)$ as 
\begin{align*}
    \left(\frac{(3r+2)^2+2}{3},\frac{(3r+2)^2-1}{3},\frac{(3r+2)^2-1}{3}\right) &=\left(\frac{(3r-1)^2+2}{3},\frac{(3r-1)^2-1}{3},\frac{(3r-1)^2-1}{3}\right)\\
    &+_H (3r-1,3r-1,3r-1) \\
    &+_H (3r+2,3r+2,3r+2).
\end{align*}
Then, by the inductive hypothesis, Lemma \ref{3k} and semigroup property, we can conclude that $g\left(\msquare_{3r+2},\msquare_{3r+2}, \left(\frac{(3r+2)^2+2}{3},\frac{(3r+2)^2-1}{3},\frac{(3r+2)^2-1}{3}\right)\right) > 0$. Thus, by the principle of mathematical induction, $g\left(\msquare_m,\msquare_m, \left(\frac{m^2+2}{3},\frac{m^2-1}{3},\frac{m^2-1}{3}\right)\right) > 0$ for every positive integer $m\ge 5$ such that $m \equiv 2 \mod 3$.
\end{proof}

\begin{lemma}\label{m=1mod3}
For any positive integer $m\ge 7$ such that $m \equiv 1 \mod 3$, the Kronecker coefficients $g\left(\msquare_m,\msquare_m, \lambda\right) > 0$ for $\lambda$ in the set \[ \left\{ \left(\frac{m^2+5}{3},\frac{m^2-1}{3},\frac{m^2-4}{3}\right),\left(\frac{m^2+5}{3},\frac{m^2+5}{3},\frac{m^2-10}{3}\right),\\ \left(\frac{m^2+5}{3},\frac{m^2+2}{3},\frac{m^2-7}{3}\right) \right\}.\]
\end{lemma}
\begin{proof}
For any positive integer $m\ge 7$ such that $m \equiv 1 \mod 3$, we can write $m = 3r+1$ for some $r \ge 2$. We will prove the proposition by induction on $r$. When $r = 2$, $3r+1 = 7$, and with the help of the computer, we can verify the statement holds true for $r= 2.$ Let $r \ge 3$, and assume that the statement is true for $r-1$. We can decompose $\msquare_m(r)$ as $$\msquare_m(r) =\left( \msquare_{m(r-1)}+_V(m(r-1)^3)\right) +_H 3^{m(r)},$$ and we can decompose the partition $\left(\frac{m(r)^2+i}{3},\frac{m(r)^2+j}{3},\frac{m(r)^2+k}{3}\right)$ as 
\begin{align*}
    \left(\frac{m(r)^2+i}{3},\frac{m(r)^2+j}{3},\frac{m(r)^2+k}{3}\right) &=\left(\frac{m(r-1)^2+i}{3},\frac{m(r-1)^2+j}{3},\frac{m(r-1)^2+k}{3}\right)\\
    &+_H (m(r-1),m(r-1),m(r-1)) \\
    &+_H (m(r),m(r),m(r)),
\end{align*}
where $(i,j,k) \in \{(5,-1,-4),(5,5,-10),(5,2,-7)\}.$
Then, by the inductive hypothesis, Lemma \ref{3k} and semigroup property, we have $g\left(\msquare_{m(r)},\msquare_{m(r)}, \left(\frac{m(r)^2+i}{3},\frac{m(r)^2+j}{3},\frac{m(r)^2+k}{3}\right)\right) > 0$, where $(i,j,k) \in \{(5,-1,-4),(5,5,-10),(5,2,-7)\}$, for any  positive integer $m\ge 7$ such that $m \equiv 1 \mod 3$
\end{proof}

\begin{lemma}\label{m=0mod3}
The Kronecker coefficient $g\left(\msquare_m,\msquare_m, \left(\frac{m^2+3}{3},\frac{m^2}{3},\frac{m^2-3}{3}\right)\right) > 0 $ and \\$g\left(\msquare_m,\msquare_m, \left(\frac{m^2+3}{3},\frac{m^2+3}{3},\frac{m^2-6}{3}\right)\right)$ for any  positive integer $m\ge 6$ such that $m \equiv 0 \mod 3$.
\end{lemma}
\begin{proof}
For any positive integer $m\ge 6$ such that $m \equiv 0 \mod 3$, we can write $m(r) = 3r$ for some $r \ge 2$. We will prove the proposition by induction on $r$.
When $r = 2$, $m(r) = 6$, and with the help of the computer, we can verify the statement holds true for $r= 2.$ Let $r \ge 3$, and assume that the statement is true for $r-1$. We can decompose $\msquare_m(r)$ as $$\msquare_m(r) =\left( \msquare_{m(r-1)}+_V(m(r-1)^3)\right) +_H 3^{m(r)},$$ and we can decompose the partition $\left(\frac{m(r)^2+3}{3},\frac{m(r)^2}{3},\frac{m(r)^2-3}{3}\right)$ as 
\begin{align*}
    \left(\frac{m(r)^2+3}{3},\frac{m(r)^2}{3},\frac{m(r)^2-3}{3}\right) &=\left(\frac{m(r-1)^2+3}{3},\frac{m(r-1)^2}{3},\frac{m(r-1)^2-3}{3}\right)\\
    &+_H (m(r-1),m(r-1),m(r-1)) \\
    &+_H (m(r),m(r),m(r)).
\end{align*} 
Then, by the inductive hypothesis, Lemma \ref{3k} and semigroup property, we can conclude that $g\left(\msquare_{m(r)},\msquare_{m(r)}, \left(\frac{m(r)^2+3}{3},\frac{m(r)^2}{3},\frac{m(r)^2-3}{3}\right)\right) > 0.$
By a completely analogous argument, we can show that $g\left(\msquare_m,\msquare_m, \left(\frac{m^2+3}{3},\frac{m^2+3}{3},\frac{m^2-6}{3}\right)\right)$ for any  positive integer $m\ge 6$ such that $m \equiv 0 \mod 3$
\end{proof}

When $m$ is even, we decompose $\msquare_m$ as $ (\msquare_{m-2}+_V(m,m))+_H(2^{m})$. By analyzing various cases based on the values and parities of $\lambda_1, \lambda_2, \lambda_3$, we are able to prove the following result:
\begin{prop}\label{even:mu3ge2}
For every even number $m\ge 6$, $g(\msquare_m, \msquare_m, \lambda)>0$ for any three-row partition $\lambda\vdash m^2$ with $\lambda_3 \ge 2$.
\end{prop}
\begin{proof}If $\lambda \vdash m^2$ is a three-row rectangular partition, then $3 \mid m$. By Corollary \ref{3krectangular}, we conclude that $g(\msquare_m, \msquare_m, \lambda)>0$. Now we assume that $\lambda$ is not a rectangular partition.

For any even integer $m \ge 6$, we can write $m = 2r$ where $r \ge 3$. We will prove this statement by induction on $r$. First, consider the base case $r=3$. In this case, $m=6$, and we can check that $g(\msquare_6,\msquare_6,\lambda)>0$ for every three-row partition $\lambda\vdash 36$ with $\lambda_3 \ge 2$ with the help of computer. 

Next, let $r\ge 4$ and assume the statement holds for $r-1$. We will prove it for $r$. By the inductive hypothesis, we assume that $g(\msquare_{2(r-1)},\msquare_{2(r-1)},\lambda)>0$ for any three-row partition $\lambda\vdash 4(r-1)^2$ with $\lambda_3 \ge 2$. Note that we can decompose $\msquare_{2r}$ as \[\msquare_{2r} = (\msquare_{2(r-1)}+_V(2r-2,2r-2))+_H(2^{2r}).\]
By Theorem \ref{thm:nn}, \[g((2^{2r}),(2^{2r}), (2a,2b,2(2r-a-b)))= g((2r,2r),(2r,2r), (2a,2b,2(2r-a-b)))>0\] for all integers $a,b$ satisfying that $0\le 2r-a-b \le b \le a\le 2r$, and \[g((2r-2,2r-2),(2r-2,2r-2), (2x,2y,2(2r-2-x-y)))>0\] for all integers $x,y$ such that $0\le 2r-2-x-y \le y \le x \le 2r-2.$

Let $\tau := (2u,2v,(8r-4)-2u-2v)$ be a non-rectangular three-row partition of $8r-4$ with all parts even. Then, we can write $(2u,2v,2w)$ as a horizontal sum of partitions $(2a,2b,2(2r-a-b))\vdash 4r$ and $(2x,2y,2(2r-2-x-y))\vdash 4r-4$, where $a = \lceil\frac{u}{2}\rceil$, $b = \lceil\frac{v}{2}\rceil$, $x = \lfloor\frac{u}{2}\rfloor$ and $y = \lfloor\frac{v}{2}\rfloor$.
Hence, it suffices to show that we can rewrite a non-rectangular three-row partition of $m^2$ as a horizontal sum of a three-row partition of $(m-2)^2$ appearing in the tensor square of $\msquare_{m-2}$ and a non-rectangular three-row partition $\tau \vdash 8r-4$ whose parts are all even. We will consider the following cases for the partition $\lambda = (\lambda_1,\lambda_2,\lambda_3)$ with $\lambda_3 \ge 2$.
\begin{itemize}
    \item [Case 1:] $\lambda_2-\lambda_3 \ge 4r-2$. In this case, we can write $\lambda$ as a horizontal sum of $(4r-2,4r-2)$ and a partition $(\lambda_1-4r+2,\lambda_2-4r+2, \lambda_3)$.
    \item [Case 2:] $\lambda_2-\lambda_3<4r-2$ and $\lambda_1-\lambda_2\ge 8r-4-4\lfloor\frac{\lambda_2-\lambda_3}{2}\rfloor$. If these conditions hold, we can define $\tau = (8r-4-2\lfloor\frac{\lambda_2-\lambda_3}{2}\rfloor, 2\lfloor\frac{\lambda_2-\lambda_3}{2}\rfloor) \vdash 8r-4$. Then, we can write $\lambda$ as a horizontal sum of $\tau$ and a three-row partition of $(m-2)^2$.
    \item [Case 3:]$\lambda_2-\lambda_3<4r-2$ and $\lambda_1-\lambda_2< 8r-4-4\lfloor\frac{\lambda_2-\lambda_3}{2}\rfloor$. In this case, we observe that $2(\lambda_2-\lambda_3)+(\lambda_1-\lambda_2) < 8r-4$ if $\lambda_2-\lambda_3$ is even, and $2(\lambda_2-\lambda_3)+(\lambda_1-\lambda_2) < 8r-2$ if $\lambda_2-\lambda_3$ is odd. Therefore, we can conclude that $\lambda_3\ge \lfloor \frac{(m-2)^2}{3} \rfloor$ under the given conditions. We further consider the following subcases:
    \begin{enumerate}
    \item If $3 \mid (m-2)^2$, then we can write $(m-2)^2 = 3k$ for some $k$ even.
    \begin{enumerate}
        \item If $\lambda$ has all parts even, then consider $\tau = \left(\lambda_1 - k,  \lambda_2 - k,\lambda_3- k\right)$. 
        \item If the parities of $\lambda_1,\lambda_2,\lambda_3$ are odd, odd, even, respectively, and $\lambda_1 > \lambda_2$, \\consider $\tau = \left(\lambda_1 - k-1,  \lambda_2 - k-1,\lambda_3- k+2\right)$.
        \item If the parities of $\lambda_1,\lambda_2,\lambda_3$ are odd, odd, even, respectively, and $\lambda_1 = \lambda_2$, then we must have $\lambda_2-\lambda_3\ge 5$ as otherwise $m^2$ or $m^2 +1$ is a multiple of $3$, which is impossible. Consider $\tau = \left(\lambda_1 - k-1,  \lambda_2 - k-1,\lambda_3- k+2\right)$.
        \item If the parities of $\lambda_1,\lambda_2,\lambda_3$ are odd, even, odd, respectively, \\consider $\tau = \left(\lambda_1 - k-1,  \lambda_2 - k,\lambda_3- k+1\right)$.
        \item If the parities of $\lambda_1,\lambda_2,\lambda_3$ are even, odd, odd, respectively, and $\lambda_2 > \lambda_3$, \\consider $\tau = \left(\lambda_1-k-2,\lambda_2-k+1,\lambda_3- k+1\right)$.
        \item If the parities of $\lambda_1,\lambda_2,\lambda_3$ are even, odd, odd, respectively, $\lambda_2= \lambda_3$, and $\lambda_1-\lambda_2\ge 5$, we consider $\tau = \left(\lambda_1-k-2,\lambda_2-k+1,\lambda_3- k+1\right)$. (Note that $\lambda_1-\lambda_2 \neq 3$ as $m \equiv 2 \mod 3$.
        \item If the parities of $\lambda_1,\lambda_2,\lambda_3$ are even, odd, odd, respectively, $\lambda_2= \lambda_3$, and $\lambda_1-\lambda_2=1$, then by Lemma \ref{m=2mod3}, we can prove the positivity of \\$g\left(\msquare_m,\msquare_m, \left(\frac{m^2+2}{3},\frac{m^2-1}{3},\frac{m^2-1}{3}\right)\right)$.
    \end{enumerate}
    \item If $(m-2)^2 \equiv 1 \bmod 3$, then we can write $(m-2)^2 = 3k+1$ for some odd integer $k$. 
    \begin{enumerate}
        \item If $\lambda$ has all parts even, consider $\tau = \left(\lambda_1 -k-1,\lambda_2 - k-1,\lambda_3- k+1\right)$. 
        
        \item If the parities of $\lambda_1,\lambda_2,\lambda_3$ are odd, odd, even, respectively, $\lambda_1 - \lambda_2 > 2$ or $\lambda_2 - \lambda_3 > 1$ , \\consider $\tau = \left(\lambda_1 - k-2,  \lambda_2 -k,\lambda_3- k+1\right)$.
        
        \item If the parities of $\lambda_1,\lambda_2,\lambda_3$ are odd, odd, even, respectively, $\lambda_1 = \lambda_2 +2 = \lambda_3+3$, then $m \equiv 1 \mod 3$. By Lemma $\ref{m=1mod3}$, we can obtain the positivity of $g\left(\msquare_m,\msquare_m, \left(\frac{m^2+5}{3},\frac{m^2-1}{3},\frac{m^2-4}{3}\right)\right)$.

        \item If the parities of $\lambda_1,\lambda_2,\lambda_3$ are odd, odd, even, respectively, $\lambda_1 = \lambda_2$ and $\lambda_2-\lambda_3 = 3$, then $3\mid m$. By Lemma \ref{m=0mod3}, we can obtain the positivity of $g\left(\msquare_m,\msquare_m, \left(\frac{m^2+3}{3},\frac{m^2+3}{3},\frac{m^2-6}{3}\right)\right)$ by semigroup property.
        
        \item If the parities of $\lambda_1,\lambda_2,\lambda_3$ are odd, odd, even, respectively, $\lambda_1 = \lambda_2$ and $\lambda_2-\lambda_3 = 5$, then $m \equiv 1 \mod 3$. By Lemma $\ref{m=1mod3}$, we can obtain the positivity of $g\left(\msquare_m,\msquare_m, \left(\frac{m^2+5}{3},\frac{m^2+5}{3},\frac{m^2-10}{3}\right)\right)$.
        
        \item If the parities of $\lambda_1,\lambda_2,\lambda_3$ are odd, odd, even, respectively, $\lambda_1 = \lambda_2$ and $\lambda_2-\lambda_3\ge 7$, consider $\tau = \left(\lambda_1 - k-2,  \lambda_2 -k-2,\lambda_3- k+3\right)$.
        
        \item If the parities of $\lambda_1,\lambda_2,\lambda_3$ are odd, even, odd, respectively, and $\lambda_2 - \lambda_3> 3$ or $\lambda_1 - \lambda_2> 1$, consider $\tau = \left(\lambda_1 - k-2,  \lambda_2 -k-1,\lambda_3- k+2\right)$.
        
        \item If the parities of $\lambda_1,\lambda_2,\lambda_3$ are odd, even, odd, respectively, and $\lambda_1 = \lambda_2 +1 = \lambda_3+4$, then $m \equiv 1 \mod 3$. By Lemma $\ref{m=1mod3}$, we can obtain the positivity of $g\left(\msquare_m,\msquare_m, \left(\frac{m^2+5}{3},\frac{m^2+2}{3},\frac{m^2-7}{3}\right)\right)$.  
        \item If the parities of $\lambda_1,\lambda_2,\lambda_3$ are odd, even, odd, respectively, $\lambda_2 - \lambda_3 = 1$ and $\lambda_1 -\lambda_2< 5$, then $\lambda_1 =\lambda_2+1$. Note that if $\lambda_1 =\lambda_2+3$, it implies that $m^2 \equiv 2 \bmod 3$, which is impossible. Thus, $3 \mid m^2$, and we can obtain the positivity of $g\left(\msquare_m,\msquare_m, \left(\frac{m^2+3}{3},\frac{m^2}{3},\frac{m^2-3}{3}\right)\right)$ by Lemma \ref{m=0mod3}.
        \item If the parities of $\lambda_1,\lambda_2,\lambda_3$ are even, odd, odd, respectively and $\lambda_2 >\lambda_3$, consider $\tau = \left(\lambda_1 - k-1,  \lambda_2 -
        k,\lambda_3- k\right)$.
        \item If the parities of $\lambda_1,\lambda_2,\lambda_3$ are even, odd, odd, respectively, $\lambda_2 = \lambda_3$ and $\lambda_1-\lambda_2 \ge 9$, consider $\tau = \left(\lambda_1 - k-5,  \lambda_2 -
        k+2,\lambda_3- k+2\right)$.
        \item If the parities of $\lambda_1,\lambda_2,\lambda_3$ are even, odd, odd, respectively, $\lambda_2 = \lambda_3$ and $\lambda_1-\lambda_2 \in \{1,7\}$, then $m \equiv 1 \mod 3$ and we can prove the positivity of $g(\msquare_m,\msquare_m,\lambda)$ by a similar argument as in the proof of Lemma \ref{m=1mod3}. 
        \item If the parities of $\lambda_1,\lambda_2,\lambda_3$ are even, odd, odd, respectively, $\lambda_2 = \lambda_3$ and $\lambda_1-\lambda_2 =3$, then $m \equiv 0 \mod 3$. We can show the positivity of $g(\msquare_m,\msquare_m,\lambda)$ by a similar argument as in the proof of Lemma \ref{m=0mod3}.
    \end{enumerate}
\end{enumerate}
\end{itemize}
For each of the cases above, $\tau$ is a non-rectangular three-row partition of $8r-4$ with all parts even, and we can write $\lambda$ as a horizontal sum of $\tau$ and a three-row partition with a long third-row of $(m-2)^2$. Then, by the semigroup property and the inductive hypothesis, we can conclude that the statement holds true for $r$. By induction, we therefore know that for $m \ge 6$, $g(\msquare_m, \msquare_m, \lambda)>0$ for any three-row partition $\lambda\vdash m^2$ with $\lambda_3 \ge 2$.
\end{proof}

Next, we will prove the positivity of $g(\msquare_m, \msquare_m, \mu)$ when $m$ is an odd integer. 
\begin{prop}\label{odd:mu3ge2part1}
For every odd integer $m\ge 5$, $g(\msquare_m, \msquare_m, \mu)>0$ for any three-row partition $\lambda\vdash m^2$ with $\mu_3 \ge 2m-1$.
\end{prop}
\begin{proof}
Let $m \ge 5$ be an odd integer. 
With the help of the computer, we can verify the statement when $ m \in \{5,7\}.$ Now consider the case where $m \ge 9.$
Note that we can decompose $\msquare_m$ as \[\msquare_m = (\msquare_{m-3} +_V (m-3)^3) +_H 3^{m},\] and $\mu$ as \[\mu = (m-3)^3 +_H (m^3) +_H (\mu_1-2m+3, \mu_2-2m+3,\mu_3-2m+3).\]
Notice that $g(\msquare_{m-3},\msquare_{m-3},(\mu_1-2m+3, \mu_2-2m+3,\mu_3-2m+3)) > 0$ by Theorem \ref{even:mu3ge2}, $g((m-3)^3,(m-3)^3,(m-3)^3)>0$ and $g(3^m,3^m,m^3) > 0 $ by Lemma \ref{3k}. Then by semigroup property, we have $g(\msquare_m, \msquare_m, \mu)>0$ for every three-row partition $\lambda\vdash m^2$ with $\mu_3 \ge 2m-1$.
\end{proof}

\begin{prop}\label{odd:mu3ge2part2}
For any odd integer $m\ge 5$, $g(\msquare_m, \msquare_m, \mu)>0$ for any three-row partition $\mu\vdash m^2$ with $2 \le \mu_3 \le 2m-2$.
\end{prop}
\begin{proof}
We can verify that the statement holds for $m \in \{5,7,9\}$ by semigroup property, together with the help of the computer.  

Let $m \ge 5$. Notice that when $\mu_1-\mu_2 \ge 2m-1$, we can decompose $\msquare_m$ as $\msquare_m = (\msquare_{m-1} +_V (m-1)) +_H 1^m$ and $\mu$ as $\mu = (m) +_H (m-1) +_H (\mu_1 - 2m+1, \mu_2,\mu_3).$ By Theorem \ref{even:mu3ge2} and the semigroup property, we conclude that $g(\msquare_m, \msquare_m, \mu)>0$ for any three-row partition $\mu\vdash m^2$ with $2 \le \mu_3 \le 2m-2$ and $\mu_1-\mu_2 \ge 2m-1$.

We shall prove the statement by induction. Let $m\ge 11$ be an odd integer. Suppose the statement is true for any odd integer less than $m$. Let $\mu \vdash m^2$ be a three-row partition such that $\mu_1-\mu_2 \le 2m-2$ and $2 \le \mu_3 \le 2m-2$. 

\begin{enumerate}
    \item [Case 1:] If $\mu_1-\mu_2 \in \{0,1\}$, consider the decomposition $\msquare_m = (\msquare_{m-4} +_V (m-4)^4) +_H 4^m$. If $(m-4)^2 \ge 3\mu_3$, we can decompose $\mu$ as \[\mu = \mu^1 +_H (2m,2m)+_H(2m-8, 2m-8),\] where $\mu^1 := \left(\left\lceil\frac{(m-4)^2-\mu_3}{2}\right\rceil, \left\lfloor\frac{(m-4)^2-\mu_3}{2}\right\rfloor, \mu_3\right).$
    Otherwise, we know that $\mu_3 \ge 16$, and we can decompose $\mu$ as \[
    \mu =  \mu^{2} +_H (2m-2,2m-2,4)  +_H (2m-8, 2m-8), \] where $\mu^{2} := \left(\left\lceil\frac{(m-4)^2-\mu_3}{2}\right\rceil+2, \left\lfloor\frac{(m-4)^2-\mu_3}{2}\right\rfloor+2, \mu_3-4\right)$. Then we have the following:
    \begin{itemize}
        \item By inductive hypothesis, we have $g\left(\msquare_{m-4},\msquare_{m-4},\mu^1 \right)>0$ and $g\left(\msquare_{m-4},\msquare_{m-4},\mu^{2} \right)>0.$
        \item By semigroup property, Theorem \ref{thm:nn} and the fact that $g((3,3,3,3),(3,3,3,3),(6,6))>0$, we have $g((m-4)^4,(m-4)^4, (2m-8, 2m-8))>0$, $g(4^m,4^m,(2m,2m)) >0$, and $g(4^m,4^m,(2m-2,2m-2,4)) >0.$
    \end{itemize} Hence, by semigroup property, we can conclude that $g(\msquare_m,\msquare_m, \mu) > 0$ when $\mu_1-\mu_2 \in \{0,1\}.$ 
    \item [Case 2:] If $\mu_1-\mu_2  \in \{2,3\}$, consider the decomposition $\msquare_m = (\msquare_{m-4} +_V (m-4)^4) +_H 4^m$. If $(m-4)^2 \ge 3(\mu_3-2)$, we can decompose $\mu$ as 
    \[\mu = \mu^1 +_H (2m,2m-2,2)  +_H (2m-8, 2m-8),\] where 
    $\mu^1 := \left(\left\lceil\frac{(m-4)^2-\mu_3}{2}\right\rceil+1, \left\lfloor\frac{(m-4)^2-\mu_3}{2}\right\rfloor+1, \mu_3-2\right)$.
    Otherwise, it implies that $\mu_3 \ge 18$ and we can decompose $\mu$ as 
    \[\mu = \mu^2 +_H (2m,2m-2,2)  +_H (2m-10, 2m-10,4),\] where 
    $\mu^2 := \left(\left\lceil\frac{(m-4)^2-\mu_3}{2}\right\rceil+3, \left\lfloor\frac{(m-4)^2-\mu_3}{2}\right\rfloor+3, \mu_3-6\right).$ Then we have the following:
    \begin{itemize}
        \item By inductive hypothesis and Theorem \ref{near two-row}, we have $g\left(\msquare_{m-4},\msquare_{m-4},\mu^1 \right)>0$ and $g\left(\msquare_{m-4},\msquare_{m-4},\mu^2 \right)>0.$
        \item By semigroup property and Theorem \ref{thm:nn}, we have  $g((m-4)^4,(m-4)^4, (2m-8, 2m-8))>0$, $g((m-4)^4,(m-4)^4, (2m-10, 2m-10,4))>0$, and $g(4^m,4^m,(2m,2m-2,2)) >0$.
    \end{itemize} Hence by semigroup property, we can conclude that $g(\msquare_m,\msquare_m, \mu) > 0$ when $\mu_1-\mu_2 \in \{2,3\}$.
    \item[Case 3:] If $a:= \mu_1-\mu_2 \ge 4$, consider the decomposition $\msquare_m = (\msquare_{m-2} +_V (m-2,m-2)) +_H 2^m$ and we will decompose $\mu$ as \begin{align*}
    \mu &= \left(\left\lceil\frac{(m-2)^2-\mu_3}{2}\right\rceil+\delta(a), \left\lfloor\frac{(m-2)^2-\mu_3}{2}\right\rfloor-\delta(a), \mu_3\right) \\
    &\quad +_H (m+1+2x, m-1-2x) \\
    &\quad +_H (m-1+2y, m-3-2y), 
    \end{align*} where $\delta(a) := 
    \begin{cases}
        0 & \text{if } a \equiv 0 \text{ or } 1 \mod 4 \\
        1 & \text{if } a \equiv 2 \text{ or } 3 \mod 4
    \end{cases}$ and $x,y$ are non-negative integers such that $4(x+y+1)= 4\left \lfloor \frac{a}{4} \right \rfloor$. Then we have the following:
    \begin{itemize}
        \item By inductive hypothesis, \[g\left(\msquare_{m-2},\msquare_{m-2},\left(\left\lceil\frac{(m-2)^2-\mu_3}{2}\right\rceil+\delta(a), \left\lfloor\frac{(m-2)^2-\mu_3}{2}\right\rfloor-\delta(a), \mu_3\right)\right)>0.\]
        \item By Theorem \ref{thm:nn}, we have $g((2m-2,2m-2),(2m-2,2m-2),(m-1+2y, m-3-2y))>0$ and $g(2^m,2^m,(m+1+2x, m-1-2x)) >0. $
    \end{itemize}
    By semigroup property, we can conclude that $g(\msquare_m,\msquare_m, \mu) > 0$ when $\mu_1-\mu_2 \ge 4.$
\end{enumerate}
Hence, by induction, for any odd integer $m \ge 5$, we have $g(\msquare_m, \msquare_m, \mu)>0$ for any three-row partition $\mu\vdash m^2$ with $2 \le \mu_3 \le 2m-2$. \end{proof}

We now have all the ingredients to prove our main theorem.
\begin{proof}[proof of Theorem \ref{mainthm}]
With the help of computer and semigroup property, we check that for $m\in \{7,9,11,13,15,17\}$, $g(\msquare_m, \msquare_m, \mu)>0$ for any three-row partition $\mu\vdash m^2$ with $2 \le \mu_3 \le 2m-2$ and $\mu_1-\mu_2 \le 2m-2$, as shown in the appendix. The result then follows from Proposition \ref{even:mu3ge2}, Proposition \ref{odd:mu3ge2part1} and Proposition \ref{odd:mu3ge2part2}.
\end{proof}

\section{Constituency of near-hooks \texorpdfstring{$(m^2-k-i,i, 1^k)$}{}}\label{sec:near-hooks}
In this section, we will discuss sufficient conditions for near-hooks to be constituents in tensor squares of square partitions

In their work \cite{ikenmeyer2017rectangular}, Ikenmeyer and Panova employed induction and the semigroup property to demonstrate the constituency of near-hooks with a second row of at most $6$ in the tensor square of a rectangle with large side lengths. 
\begin{thm}[\cite{ikenmeyer2017rectangular} Corollary 4.6]\label{rect-near-hooks} Fix $w \ge h \ge 7$. We have that $g(\lambda, h\times w, h \times w)>0$ for all $\lambda = (hw-j-|\rho|,1^j+_H \rho)$ with $\rho \neq \emptyset$ and $|\rho| \le 6$ for all $j \in [1,h^2-R_\rho]$ where $R_\rho = |\rho|+\rho_1+1$, except in the following cases: $\lambda \in \{(hw-3,2,1), (hw-h^2+3, 2, 1^{h^2-5}), (hw-4,3,1), (hw-h^2+3,2,2,1^{h^2-7})\}.$ \end{thm}

The positivity of certain classes of near-hooks can be directly derived from Theorem \ref{rect-near-hooks}.
\begin{cor}\label{cor:square-near-hooks}
Let $m \ge 7$. For all $\mu_i(k,m) = (m^2-k-i,i,1^k)$ with $i \in [2,7]$ and $k \in [$, $g(\msquare_m,\msquare_m, \mu_i(k,m)) > 0$ except in the following cases: (1) $i = 2$ with $k = 1$ or $k = m^2-5$, (2) $i = 3$ and $k = 1$.
\end{cor}

We present an alternative proof approach for finding sufficient conditions for two classes of near-hooks to be constituents in tensor squares of square partitions using a tool aimed specifically at the Saxl conjecture was developed in \cite{pak2013kronecker} as follows:
\begin{thm}[\cite{pak2013kronecker} Main Lemma] \label{character} Let $\mu = \mu'$ be a self-conjugate partition of $n$, and let $\nu = (2\mu_1-1,2\mu_2-3,2\mu_3-5,\dots) \vdash n$ be the partition whose parts are lengths of the principal hooks of $\mu$. Suppose $\chi^{\lambda}[\nu] \neq 0$ for some $\lambda \vdash n$. Then $\chi^{\lambda}$ is a constituent of $\chi^{\mu}\bigotimes \chi^{\mu}.$
\end{thm}

Let $\mu_i(k,m) := (m^2-k-i,i, 1^k)$ and $\alpha_m = (2m-1,2m-3,\dots,1).$ By Theorem \ref{character}, that $|\chi^{\mu(k,m)}(\alpha_m)| \neq 0$ would imply $g(\msquare_m,\msquare_m, \mu_i(k,m)) >0$. In particular, we will discuss the number of rim-hook tableaux of shape $\mu_2(k,m) = (m^2-k-2,2, 1^k)$ and weight $\alpha_m$.

To use the Murnaghan-Nakayama rule to compute the characters, we consider the construction of an arbitrary rim-hook tableau of shape $\mu_2(k,m)$ and weight $(1,3,\dots,2m-1).$ Observe that the $1-$hook can only be placed at the upper left corner, and there are three ways to place the $3-$hook, as illustrated in the following diagrams: 
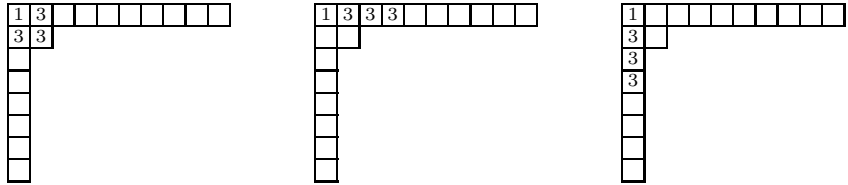
\vspace{.2in}

\begin{center}
\ytableausetup{smalltableaux}
\ytableaushort
  {13\none,33, \none}
 * {10,2,1,1,1,1,1,1}
 \hspace{1cm}
\ytableaushort
  {1333\none,\none,\none}
 * {10,2,1,1,1,1,1,1}
  \hspace{1cm}
\ytableaushort
  {1\none,3\none,3\none,3\none}
 * {10,2,1,1,1,1,1,1}
\end{center}
\vspace{.2in}

Let $P_{R(m)}(k)$ denote the number of partitions of $k$ whose parts are  distinct odd integers from the set $${R(m)} = \{5,7,\dots,2m-1\}.$$  Observing the diagrams above, we can deduce that the height of any rim-hook tableau with the shape $\mu_2(k,m)$ and weight $(1,3,\dots,2m-1)$ is always an odd number. From left to right,  the quantities of rim-hook tableaux corresponding to the three diagrams are $P_{R(m)}(k)$, $P_{R(m)}(k+2)$, and $P_{R(m)}(k-2)$, respectively. Thus, by Murnaghan-Nakayama rule, we therefore have \[\chi^{(n-k-2,2,1^k)}(2m-1,2m-3,\dots,3,1) = -P_{R(m)}(k)-P_{R(m)}(k+2)-P_{R(m)}(k-2).\] Thus, $g(\lambda_m,\lambda_m, (n-k-2,2,1^k)) > 0$ if $P_{R(m)}(k)+P_{R(m)}(k+2)+P_{R(m)}(k-2) > 0$, which is equivalent to that \[\max \{P_{R(m)}(k),P_{R(m)}(k+2),P_{R(m)}(k-2)\}> 0.\] 

\begin{lemma}\label{lem:casei=2} Let $m \ge 8$ be fixed, $0 \le k \le m^2-4$ and let \[NK_2(m) = \{1,2,4,6,8,m^2-12,m^2-10, m^2-8,m^2-6, m^2-5\}.\] Then $P_{R(m)}(k)+P_{R(m)}(k+2)+P_{R(m)}(k-2) > 0$ if and only if  $k \notin NK_2(m)$. \end{lemma}

\begin{proof} By directly checking the values for $k \in \{1,2,4,6,8\}$, we find that \[P_{R(m)}(k)+P_{R(m)}(k+2)+P_{R(m)}(k-2) = 0\]  holds true. Note the sum of elements in $R(m)$ is $m^2-4$ and therefore $P_{R(m)}(k) = P_{R(m)}(m^2-4-k)$. It follows that if $k \in \{m^2-12,m^2-10, m^2-8,m^2-6, m^2-5\}$, then $P_{R(m)}(k)+P_{R(m)}(k+2)+P_{R(m)}(k-2) = 0.$ 

We shall prove the other direction by induction on $m$. We can check the statement is true for $m = 8.$ Now, assuming that the statement is true for $m \ge 8$, we will show that it holds true for $m+1.$  Due to the symmetry of $P_{R(m+1)}(k)$, it suffices to demonstrate that $P_{R(m+1)}(k)+P_{R(m+1)}(k+2)+P_{R(m+1)}(k-2) > 0$ for any $k \in [\lceil\frac{(m+1)^2-4}{2}\rceil] \setminus \{1,2,4,6,8\}$.

Since $R(m)\subset R(m+1)$ by construction, we can assert that \[P_{R(m+1)}(k)+P_{R(m+1)}(k+2)+P_{R(m+1)}(k-2) \ge P_{R(m)}(k)+P_{R(m)}(k+2)+P_{R(m)}(k-2) > 0\] for any $k \in [m^2-4] \setminus NK_2(m)$ by inductive hypothesis. It is easy to see that $\lceil\frac{(m+1)^2-4}{2}\rceil < m^2-12$ when $m \ge 6$.

Then by the inductive hypothesis, it follows that $P_{R(m+1)}(k)+P_{R(m+1)}(k+2)+P_{R(m+1)}(k-2) > 0$ for any $k \notin NK_2(m+1)$, which completes the induction.
\end{proof}



\begin{thm}\label{mu2hook} Let $m \ge 8$ be fixed and $0 \le k \le m^2-4$. Then, $g(\msquare_m,\msquare_m, \mu_2(k,m)) > 0$ if and only if $k \notin \{1, m^2-5\}$.\end{thm}

\begin{proof} ($\Rightarrow$) If $k = 1$, by Corollary \ref{cor:zerocase}, $g(\msquare_m,\msquare_m, (m^2-3,2,1)) = 0$. If $k = m^2-5$, by the transposition property, $g(\msquare_m,\msquare_m, (m^2-k-2,2,1^k)) = 0$.

($\Leftarrow$) If $k \notin NK_2(m)$, the result follows from Theorem \ref{character}, Murnaghan-Nakayama rule, and Lemma \ref{lem:casei=2}.
If $k = 2,4,6,8$, we consider the decomposition $(m^2-4,2,1^2) = (21,2,1,1)+_H(m^2-25)$, $(m^2-6,2,1^4) = (19,2,1^4)+_H(m^2-25)$, $(m^2-8,2,1^6) = (17,2,1^6)+_H(m^2-25)$,  $(m^2-10,2,1^8) = (15,2,1^8)+_H(m^2-25)$, respectively. Then by the semigroup property, it follows that  $g(\msquare_m,\msquare_m, (m^2-k,2,1^k)) > 0$ for $k \in \{2,4,6,8\}$. By the transposition property of Kronecker coefficients, $g(\msquare_m,\msquare_m, (m^2-k,2,1^k)) > 0$ for $k \in \{m^2-6,m^2-8,m^2-10,m^2-12\}$.
\end{proof}

Similarly, there are only two ways to place the 1-hook and 3-hook into a rim-hook tableau of shape $\mu_3(k,m)$ and weight $\alpha_m$, as illustrated below. 
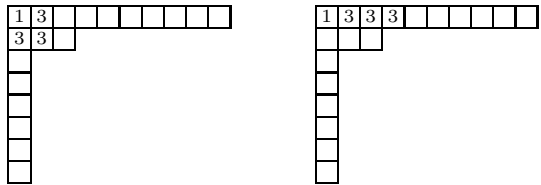
\vspace{.1in}

\begin{center}
\ytableausetup{smalltableaux}
\ytableaushort
  {13\none,33, \none}
 * {10,3,1,1,1,1,1,1}
 \hspace{1cm}
\ytableaushort
  {1333\none,\none,\none}
 * {10,3,1,1,1,1,1,1}
  \hspace{1cm}
\end{center}
\vspace{.2in}
Therefore, we have \[\chi^{(n-k-3,3,1^k)}(2m-1,2m-3,\dots,3,1) = P_{R(m)}(k)+P_{R(m)}(k+3).\] It follows that $g(\msquare_m,\msquare_m, (n-k-3,3,1^k)) > 0$ if $P_{R(m)}(k)+P_{R(m)}(k+3) > 0$.

\begin{lemma}\label{lem:casei=3} Let $m \ge 5$ be fixed, $0\le k \le m^2-7$ and let \[NK_3(m) := \{1,3,m^2-10, m^2-8\}.\] Then $P_{R(m)}(k)+P_{R(m)}(k+3) > 0$ if and only if   $k \notin NK_3(m)$. \end{lemma}
\begin{proof} By directly checking the values for $k \in \{1,3\}$, we find that $P_{R(m)}(k)+P_{R(m)}(k+3) = 0$ holds true. Note the sum of elements in $R(m)$ is $m^2-4$ and therefore $P_{R(m)}(k) = P_{R(m)}(m^2-4-k)$. It follows that if $k \in \{m^2-8, m^2-10\}$, then $P_{R(m)}(k)+P_{R(m)}(k+3) = 0$.

We shall prove the other direction by induction on $m$. It is easy to check that the statement is true for $m = 7.$ Now, assuming that the statement is true for $m \ge 7$, we will show that it is also true for $m+1.$ Due to the symmetry of $P_{R(m+1)}(k)$, it suffices to demonstrate that $P_{R(m+1)}(k)+P_{R(m+1)}(k+3) > 0$ for any $k \in [\lfloor\frac{(m+1)^2-7}{2}\rfloor] \setminus \{1,3\}$.

Since $R(m)\subset R(m+1)$ by construction, we can assert that \[P_{R(m+1)}(k)+P_{R(m+1)}(k+3) \ge P_{R(m)}(k)+P_{R(m)}(k+3) > 0\] for any $k \in [m^2-4] \setminus NK_3(m)$ by inductive hypothesis. We can verify that $\lfloor\frac{(m+1)^2-7}{2}\rfloor < m^2-10$ when $m \ge 5$.
Then by the inductive hypothesis, we conclude that $P_{R(m)}(k)+P_{R(m)}(k+3)> 0$ for any $k \notin NK_3(m+1)$, which completes the induction.
\end{proof}

\begin{thm}\label{mu3hook} 
Let $m\ge 7$ be fixed and $0 \le k \le m^2-6$. Then, $g(\msquare_m,\msquare_m, \mu_3(k,m))> 0$ if and only if $k \neq 1$.\end{thm}
\begin{proof} ($\Rightarrow$) If $k = 1$, by Corollary \ref{cor:zerocase}, $g(\msquare_m,\msquare_m, (m^2-k-3,3,1^k)) = 0$.

($\Leftarrow$) Now assume that $k \in \{0,2,3,\dots,m^2-6\}$.
If $k = 3$, we can decompose the partition $(m^2-6,3,1^3)$ as $(3,3,1,1,1)+_H (m^2-9)$. Since $g((3,3,3),(3,3,3),(3,3,1,1,1))>0$, by semigroup property, it follows that $g(\msquare_m,\msquare_m, (m^2-6,3,1^3)) > 0.$ If $k = m^2-10$, we can decompose the partition $(7,3,1^k)$ as $(7,3,1^6)+_V (1^{(m^2-16)})$. Since $g((4,4,4,4),(4,4,4,4),(7,3,1^6))>0$, by semigroup property,it follows that $g(\msquare_m,\msquare_m, (7,3,1^k)) > 0.$ If $k = m^2-8$, we can decompose the partition $(5,3,1^k)$ as $(5,3,1^8)+_V (1^{(m^2-16)})$. Since $g((4,4,4,4),(4,4,4,4),(5,3,1^8))>0$, by semigroup property, it follows that $g(\msquare_m,\msquare_m, (5,3,1^k)) > 0.$

If $k = m^2-6$, $g(\msquare_m,\msquare_m, (3,3,1^k)) = g(\msquare_m,\msquare_m, (m^2-4,2,2))> 0$ by Theorem \ref{mainthm}.
If $k < m^2-6$ and $k \notin NK_3(m)$, the result follows from Theorem \ref{character}, Murnaghan-Nakayama rule, and Lemma \ref{lem:casei=3}.
\end{proof}



Next, we will discuss the constituency of near-hooks with a second row of length of at least 8.
\begin{prop}\label{general statement} For every $i \ge 8$, we have $g(\msquare_m,\msquare_m, \mu_i(k,m)) > 0$ for all $m \ge 20$ and $k \ge 0$.\end{prop}
\begin{proof}
Let $i \ge 8$ be fixed. Suppose that $m \ge 20$.

If $k\ge 7m+9-i$, we can decompose the transpose of partition $\mu_i(k,m)$, that is $(k+2, 2^{i-1}, 1^{m^2-2i-k})$ as  $(k+2, 2^{i-1}, 1^{m^2-2i-k}) = (k_1, 1^{i-1})+_H (k+2-k_1, 1^{m^2-i-k-1})$ where $k_1 = 7m-i+1$. Since $k\ge 7m+9-i$, we have $k+2-k_1 \ge 10$. Then by Theorem \ref{thm:hookpos}, we have $g(7^m, 7^m, (k_1, 1^{i-1})) > 0$ and $g ((m-7)^m, (m-7)^m, (k+2-k_1, 1^{m^2-i-k-1})) > 0$. We can use the Semigroup property to add the partition triples, which implies that $g(\msquare_m, \msquare_m, (k+2, 2^{i-1}, 1^{m^2-2i-k}))>0$. Then by the transposition property, we have $g(\msquare_m, \msquare_m, \mu_i(k,m))>0$.

If $k\le 7m+8-i$, $2i+k \le 2(7m+8) \le 15m$, we consider the decomposition $\msquare_m = \msquare_{m_1}+_V (m_1^{m-m_1}) +_H ((m-m_1)^m)$, where $m_1 = \left \lceil \sqrt{k+8} \,\right \rceil$. Since $m_1 \le \left \lceil \sqrt{k+8} \,\right \rceil$ and $m \ge 16$, we have $m_1^2- (k+8) \le k+8 \le m^2-k-2i$, which implies that $m^2-m_1^2-i+4\ge i-4$. Moreover, since $k \le 7m$ and $m\ge 20$, we have $m_1 \le m-8$. We will show that there exists a decomposition $\mu_i(k,m) = \mu_4(k,m_1) +_H (a+d_1, a) +_H (b+d_2,b)$ such that $(a+d_1, a) \vdash m_1(m-m_1)$, $ (b+d_2,b) \vdash m(m-m_1)$ and $a+b = i-4$.  We consider the following two cases:
\begin{enumerate}
    \item [Case 1:] If $m$ is odd, then $m(m-m_1)$ and $ m^2-m_1^2-2(i-4)$ always have the same parity. If $ m^2-m_1^2-2(i-4) = m(m-m_1)-2$, let $d_2 = m(m-m_1)-4$ and it follows that $b = 2$; otherwise, let $d_2 = \min (m(m-m_1), m^2-m_1^2-2(i-4))$. It is easy to check that $a \neq 1$ and $b \neq 1$ in this case.
    \item [Case 2:] If $m$ is even, then $m_1(m-m_1)$ and $ m^2-m_1^2-2(i-4)$ always have the same parity. If $ m^2-m_1^2-2(i-4) = m_1(m-m_1)-2$, let $d_1 = m_1(m-m_1)-4$ and it follows that $a  =2$; otherwise, let $d_1= \min (m_1(m-m_1), m^2-m_1^2-2(i-4))$. It is easy to check that $a \neq 1$ and $b \neq 1$ in this case.
\end{enumerate}
By Corollary \ref{cor:square-near-hooks}, we have $g(\msquare_{m_1}, \msquare_{m_1}, \mu_4(k,m_1) >0$. Since $m, m_1, m-m_1 \ge 8$, by Theorem \ref{t:strict}, we can conclude that $g((m_1^{m-m_1}),(m_1^{m-m_1}),(a+d_1,a) )> 0$ and $g(((m-m_1)^m),((m-m_1)^m),(b+d_2,b)) > 0$. Then, adding the partition triples horizontally by semigroup property, we can conclude that $g(\msquare_m, \msquare_m, \mu_i(k,m))>0$ for every $m \ge 20.$
\end{proof}
\begin{cor} \label{cor:conclusion-near-hooks}For every $i \ge 8$, we have $g(\msquare_m,\msquare_m, \mu_i(k,m)) > 0$ for all $m \ge 20$ and $k \ge 0$.
\end{cor}
\begin{proof}
It follows directly from Corollary \ref{cor:square-near-hooks} and Proposition \ref{general statement}.
\end{proof}
    
\section{Additional Remarks}\label{sec:final}
\subsection{} We have proved the positivity of Kronecker coefficients indexed by pairs of rectangular Young diagrams and certain three-row partitions of special shapes. We could further use the result of square Kronecker coefficients to investigate the behavior of tensor squares of irreducible representations for rectangular Young diagrams and explore the positivity properties for specific families of rectangular partitions. 

\subsection{}Since the decomposition of a rectangular partition can only be achieved by writing it as a horizontal or vertical sum of two rectangular partitions, it limits the application of the semigroup property. For partitions with more rows or larger Durfee size, there are instances where the semigroup property fails to prove positivity. A specific example is the Kronecker coefficient $g(\msquare_m,\msquare_m,((m+1)^{m-1}, 1))$. Due to the partition shapes involved, the only valid method to decompose them to satisfy number-theoretical conditions is as follows:
$\msquare_m = ((m-1)^m )+_H (1^m)= (m^{m-1}) +_V (m)$ and 
$((m+1)^{m-1}, 1) = (m^{m-1})+_H (1^m).$ 
However, $g(1^m,1^m,1^m) = g (m,m,1^m) = 0$, indicating that we cannot rely on this approach to prove positivity. This demonstrates the limitations of the semigroup property in certain cases. 

When $m$ is even, we can establish through a recursive argument that there exists no rim-hook tableau of shape \(((m+1)^{m-1}, 1)\) with type \(\alpha_m\). Note that there is a unique arrangement for both the \((2m-1)\)-hook and the \((2m-3)\)-hook. These two longest rim-hooks invariably occupy the skew-shape \(((m+1)^{m-1}, 1)/((m-1)^{m-3}, 1)\), as depicted in the diagram below. Then the problem is reduced to a search for a rim-hook tableau with shape \(((m-1)^{m-3}, 1)\) and type \(\alpha_{m-2}\). By iterating this process, we know that a rim-hook tableau with shape \(((m+1)^{m-1}, 1)\) and type \(\alpha_m\) exists if and only if a rim-hook tableau with shape \((5, 4)\) and type \((5, 3, 1)\) can be found. Therefore, there does not exist a rim-hook tableau of shape $((m+1)^{m-1}, 1)$ and type $\alpha_m$, which implies that $\chi^{((m+1)^{m-1}, 1)}(\alpha_m) = 0$ by Murnaghan-Nakayama Rule. Hence, the character approach (Theorem \ref{character}) is also not applicable in this case.

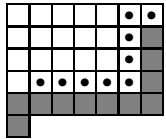
\begin{center}
\ydiagram 
[*(white)]{5,5,5,1}
*[ *(white)\bullet]{7,6,6,6}
*[*(gray)]{7,7,7,7,7,1}
\end{center}



\appendix 
\section{Missing partitions in tensor square of square with a small side length}

With the help of computer, we find all partitions $\lambda \vdash m^2$ such that $g(\msquare_m, \msquare_m, \lambda) = 0$ for $m = 4,5,6,7$:
\begin{itemize}
    \item  $g(\msquare_4,\msquare_4,\lambda) = 0$ if and only if $\lambda$ or $\lambda^\prime \in \seqsplit{\{(15,1), (14,1,1),(13,2,1),(12,3,1),(12,1,1,1,1),(11,5),(10,1,1,1,1,1,1),(9,7),(8,7,1),(8,2,1,1,1,1,1,1),(7,7,2),(7,5,4)\}}$;
    \item  $g(\msquare_5,\msquare_5,\lambda) = 0$ if and only if $\lambda$ or $\lambda^\prime \in \seqsplit{ \{(24,1),(23,1,1),(22,2,1),(21,3,1),(21,1,1,1,1),(19,1,1,1,1,1,1),(14, 1, 1, 1, 1, 1, 1, 1, 1, 1, 1, 1)\}}$;
    \item $g(\msquare_6,\msquare_6,\lambda) = 0$ if and only if $\lambda$ or  $\lambda^\prime \in \seqsplit{\{(35, 1),(34, 1, 1),(33, 2, 1), (32, 3, 1), (32, 1, 1, 1, 1), (30, 1, 1, 1, 1, 1, 1),}  {(23, 1^{13}), (19, 17)\}}$;
    \item $g(\msquare_7,\msquare_7,\lambda) = 0$ if and only if $\lambda$ or  $\lambda^\prime \in \seqsplit{\{\{(48, 1),(47, 1, 1),(46, 2, 1), (45, 3, 1), (45, 1, 1, 1, 1), (43, 1, 1, 1, 1, 1, 1)\}}.$
\end{itemize}

\printbibliography

@article{Ikenmeyer_2015,
   title={The Saxl conjecture and the dominance order},
   volume={338},
   ISSN={0012-365X},
   url={http://dx.doi.org/10.1016/j.disc.2015.04.027},
   DOI={10.1016/j.disc.2015.04.027},
   number={11},
   journal={Discrete Mathematics},
   publisher={Elsevier BV},
   author={Ikenmeyer, Christian},
   year={2015},
   month={11},
   pages={1970–1975}
}

@article{Luo_2016,
   title={The Saxl conjecture for fourth powers via the semigroup property},
   volume={45},
   ISSN={1572-9192},
   url={http://dx.doi.org/10.1007/s10801-016-0700-z},
   DOI={10.1007/s10801-016-0700-z},
   number={1},
   journal={Journal of Algebraic Combinatorics},
   publisher={Springer Science and Business Media LLC},
   author={Luo, Sammy and Sellke, Mark},
   year={2016},
   month={8},
   pages={33–80}
}

@article{pak2013kronecker,
title = {Kronecker products, characters, partitions, and the tensor square conjectures},
journal = {Advances in Mathematics},
volume = {288},
pages = {702-731},
year = {2016},
issn = {0001-8708},
doi = {https://doi.org/10.1016/j.aim.2015.11.002},
url = {https://www.sciencedirect.com/science/article/pii/S0001870815004508},
author = {Igor Pak and Greta Panova and Ernesto Vallejo}
}

@misc{dou2009hive,
      title={A hive model determination of multiplicity-free Schur function products and skew Schur functions}, 
      author={Donna Q. J. Dou and Robert L. Tang and Ronald C. King},
      year={2009},
      eprint={0901.0186},
      archivePrefix={arXiv},
      primaryClass={math.CO}
}

@article{Gutschwager_2010,
   title={On Multiplicity-Free Skew Characters and the Schubert Calculus},
   volume={14},
   ISSN={0219-3094},
   url={http://dx.doi.org/10.1007/s00026-010-0063-4},
   DOI={10.1007/s00026-010-0063-4},
   number={3},
   journal={Annals of Combinatorics},
   publisher={Springer Science and Business Media LLC},
   author={Gutschwager, Christian},
   year={2010},
   month={9},
   pages={339–353}
}

@article{thomas_yong_2010, title={Multiplicity-Free Schubert Calculus}, volume={53}, DOI={10.4153/CMB-2010-032-x}, number={1}, journal={Canadian Mathematical Bulletin}, publisher={Cambridge University Press}, author={Thomas, Hugh and Yong, Alexander}, year={2010}, pages={171–186}}

@article{Pak_2013,
   title={Strict unimodality of q-binomial coefficients},
   volume={351},
   ISSN={1631-073X},
   url={http://dx.doi.org/10.1016/j.crma.2013.06.008},
   DOI={10.1016/j.crma.2013.06.008},
   number={11-12},
   journal={Comptes Rendus Mathematique},
   publisher={Elsevier BV},
   author={Pak, Igor and Panova, Greta},
   year={2013},
   month={6},
   pages={415–418}
}

@article{ikenmeyer2017rectangular,
title = {Rectangular Kronecker coefficients and plethysms in geometric complexity theory},
journal = {Advances in Mathematics},
volume = {319},
pages = {40-66},
year = {2017},
issn = {0001-8708},
doi = {https://doi.org/10.1016/j.aim.2017.08.024},
url = {https://www.sciencedirect.com/science/article/pii/S0001870816304534},
author = {Christian Ikenmeyer and Greta Panova}
}

@article{liu2014simplified,
 ISSN = {00029939, 10886826},
 URL = {https://www.jstor.org/stable/90013039},
 author = {Ricky Ini Liu},
 journal = {Proceedings of the American Mathematical Society},
 number = {9},
 pages = {pp. 3657--3664},
 publisher = {American Mathematical Society},
 title = {A simplified Kronecker rule for one hook shape},
 volume = {145},
 year = {2017}
}

@article{LI2021112340,
title = {Saxl Conjecture for triple hooks},
journal = {Discrete Mathematics},
volume = {344},
number = {6},
pages = {112340},
year = {2021},
issn = {0012-365X},
doi = {https://doi.org/10.1016/j.disc.2021.112340},
url = {},
author = {Xin Li}
}

@article{Christandl2007,
author = {Christandl, Matthias and Harrow, Aram W and Mitchison, Graeme},
doi = {10.1007/s00220-006-0157-3},
issn = {1432-0916},
journal = {Communications in Mathematical Physics},
number = {3},
pages = {575--585},
title = {{Nonzero Kronecker Coefficients and What They Tell us about Spectra}},
url = {https://doi.org/10.1007/s00220-006-0157-3},
volume = {270},
year = {2007}
}

@article{10.36045/bbms/1103408635,
author = {Jeffrey B. Remmel and Tamsen Whitehead},
title = {{On the Kronecker product of Schur functions of two row shapes}},
volume = {1},
journal = {Bulletin of the Belgian Mathematical Society - Simon Stevin},
number = {5},
publisher = {The Belgian Mathematical Society},
pages = {649 -- 683},
year = {1994},
doi = {10.36045/bbms/1103408635},
URL = {https://doi.org/10.36045/bbms/1103408635}
}

@article{Ikenmeyer_2017,
   title={On vanishing of Kronecker coefficients},
   volume={26},
   ISSN={1420-8954},
   url={http://dx.doi.org/10.1007/s00037-017-0158-y},
   DOI={10.1007/s00037-017-0158-y},
   number={4},
   journal={computational complexity},
   publisher={Springer Science and Business Media LLC},
   author={Ikenmeyer, Christian and Mulmuley, Ketan D. and Walter, Michael},
   year={2017},
   month={7},
   pages={949–992}
}

@article{Ballantine2005OnTK,
  title={On the Kronecker Product s(n-p,p) * s$\lambda$},
  author={Cristina M. Ballantine and Rosa C. Orellana},
  journal={Electron. J. Comb.},
  year={2005},
  volume={12}
}

@article{blasiak2012kronecker,
	author = {Jonah Blasiak},
	journal = {S{\'e}minaire Lotharingien de Combinatoire},
	month = {9},
	title = {Kronecker coefficients for one hook shape},
	volume = {77},
	year = {2017}}

@article{Remmel1989AFF,
  title={A formula for the Kronecker products of Schur functions of hook shapes},
  author={Jeffrey B. Remmel},
  journal={Journal of Algebra},
  year={1989},
  volume={120},
  pages={100-118}
}

@article{pak2015complexity,
	author = {Pak, Igor and Panova, Greta},
	journal = {computational complexity},
	number = {1},
	pages = {1--36},
	title = {On the complexity of computing Kronecker coefficients},
	volume = {26},
	year = {2017}}

@article{article,
author = {B\"urgisser, Peter and Ikenmeyer, Christian},
year = {2008},
month = {01},
pages = {},
title = {The complexity of computing Kronecker coefficients},
volume = {DMTCS Proceedings vol.},
journal = {FPSAC'08 - 20th International Conference on Formal Power Series and Algebraic Combinatorics},
doi = {10.46298/dmtcs.3622}
}

@book{stanley1997enumerative,
  title={Enumerative Combinatorics: Volume 2},
  author={Stanley, R.P. and Fomin, S.},
  isbn={9780521789875},
  lccn={96044267},
  series={Cambridge Studies in Advanced Mathematics},
  url={https://books.google.com/books?id=2QOtugEACAAJ},
  year={1997},
  publisher={Cambridge University Press}
}

@book{books/daglib/0077285,
  author = {Sagan, Bruce E.},
  isbn = {978-0-534-15540-7},
  pages = {I-XV, 1-197},
  publisher = {Wadsworth},
  series = {Wadsworth \& Brooks / Cole mathematics series},
  title = {The symmetric group - representations, combinatorial algorithms, and symmetric functions.},
  year = 1991
}

@article{HSTZ,
   title={Conjugacy action, induced representations and the Steinberg square for simple groups of Lie type},
   volume={106},
   ISSN={0024-6115},
   url={http://dx.doi.org/10.1112/plms/pds062},
   DOI={10.1112/plms/pds062},
   number={4},
   journal={Proceedings of the London Mathematical Society},
   publisher={Wiley},
   author={Heide, Gerhard and Saxl, Jan and Tiep, Pham Huu and Zalesski, Alexandre E.},
   year={2012},
   month={11},
   pages={908–930}
}

@incollection{stanley2000positivity,
  author    = {Stanley, Richard P.},
  title     = {Positivity problems and conjectures in algebraic combinatorics},
  booktitle = {Mathematics: Frontiers and Perspectives},
  pages     = {295--319},
  publisher = {American Mathematical Society},
  address   = {Providence, RI},
  year      = {2000}
}

@misc{panova2023complexity,
      title={Complexity and asymptotics of structure constants}, 
      author={Greta Panova},
      year={2023},
      eprint={2305.02553},
      archivePrefix={arXiv},
      primaryClass={math.CO}
}

@article{Littlewood1958,
  author = {Dudley E. Littlewood},
  title = {Products and plethysms of characters with orthogonal, symplectic and symmetric groups},
  journal = {Canadian Journal of Mathematics},
  volume = {10},
  year = {1958},
  pages = {17--32},
}

@article{VALLEJO2014243,
title = {A diagrammatic approach to Kronecker squares},
journal = {Journal of Combinatorial Theory, Series A},
volume = {127},
pages = {243-285},
year = {2014},
issn = {0097-3165},
doi = {https://doi.org/10.1016/j.jcta.2014.06.002},
url = {https://www.sciencedirect.com/science/article/pii/S0097316514000880},
author = {Ernesto Vallejo}
}

@article{pak14unimodality,
	author = {Pak, Igor and Panova, Greta},
	journal = {Journal of Algebraic Combinatorics},
	number = {4},
	pages = {1103--1120},
	title = {Unimodality via Kronecker products},
	volume = {40},
	year = {2014}}
\end{document}